\RequirePackage{rotating}
\documentclass[11pt]{amsart}
\usepackage[table]{xcolor}
\usepackage[utf8]{inputenc}
\usepackage[T1]{fontenc}
\usepackage[a4paper,text={460pt,660pt},headsep=8mm, centering,marginparwidth=2cm]{geometry}
\usepackage[hidelinks,pdfusetitle]{hyperref}
\usepackage[activate={true,nocompatibility},final,tracking=true,kerning=true,factor=1100,stretch=10,shrink=10]{microtype}
\usepackage{braket}    
\usepackage{amsthm}    
\usepackage{amssymb}   
\usepackage{mathrsfs}  
\usepackage[shortlabels]{enumitem}
\usepackage{caption}
\usepackage[margin=5pt,justification=centering,labelformat=simple,labelfont=sc]{subcaption}
\usepackage[mode=buildnew]{standalone} 
\usepackage{tikz-cd}
\usepackage{pgfplots}
\pgfplotsset{compat=newest}

\usepackage{bbold}

\theoremstyle{plain}
\newtheorem*{theo*}{Theorem}
\newtheorem{theorem}{Theorem}[section]
\newtheorem{corollary}[theorem]{Corollary}
\newtheorem{proposition}[theorem]{Proposition}
\newtheorem{lemma}[theorem]{Lemma}

\newtheorem{definition}[theorem]{Definition}
\newtheorem{hyp}{Assumption}
\newtheorem{ichyp}{Assumption}
\newenvironment{chyp}[1]
  {\renewcommand\theichyp{#1}\ichyp}
  {\endichyp}
\newtheorem{icprop}{Proposition}
\newenvironment{cprop}[1]
  {\renewcommand\theicprop{#1}\icprop}
  {\endicprop}

\theoremstyle{remark}
\newtheorem{remark}[theorem]{Remark}
\newtheorem*{rem*}{Remark}

\newcommand{\cf}{\ensuremath{\mathscr{F}}}
\newcommand{\cg}{\ensuremath{\mathscr{G}}}

\newcommand{\ind}[1]{\mathbb{1}_{#1}}
\newcommand{\un}{\mathbb{1}}
\newcommand{\zero}{\mathbb{0}}
\newcommand{\ca}{\ensuremath{\mathscr{A}}}
\newcommand{\cl}{\ensuremath{\mathscr{L}^\infty}}

\usepgfplotslibrary{fillbetween}
\usepackage[backend=biber,
	    url=false,
	    doi=false,
	    isbn=false,
	    giveninits=true,
	    date=year,
	    maxbibnames=99,
	    sortcites=true,
	    sorting=nty,
	    style=numeric]{biblatex}
\renewbibmacro{in:}{}
\usepackage{bm} 
\newcommand{\kkk}{{\ensuremath{\bm{\kk}}}}

\date{\today}
\author{Jean-François Delmas}
\address{Jean-François Delmas,
  CERMICS, \'{E}cole nationale des ponts et chauss\'ees, 77455
  Marne-la-Vall\'ee, France} 
\email{jean-francois.delmas@enpc.fr}

\author{Dylan Dronnier}
\address{Dylan Dronnier,
  CERMICS, \'{E}cole nationale des ponts et chauss\'ees, 77455
  Marne-la-Vall\'ee, France} 
\email{dylan.dronnier@enpc.fr}

\author{Pierre-André Zitt}
\address{Pierre-André Zitt, LAMA, Université Gustave Eiffel, 77420
Champs-sur-Marne, France}
\email{pierre-andre.zitt@univ-eiffel.fr}

\newcommand{\norm}[1]{\left\lVert\,#1\,\right\rVert}

\newcommand{\R}{\ensuremath{\mathbb{R}}}

\newcommand{\N}{\ensuremath{\mathbb{N}}}

\newcommand{\cpa}{\mathcal{P}^\mathrm{Anti}}
\newcommand{\F}{\ensuremath{\mathcal{F}}}
\newcommand{\AF}{\ensuremath{\mathcal{F}^\mathrm{Anti}}}
\newcommand{\kk}{\ensuremath{\mathrm{k}}}
\newcommand{\rd}{\ensuremath{\mathrm{d}}}

\newcommand{\cp}{\ensuremath{\mathcal{P}}}
\newcommand{\Tinf}{\ensuremath{\mathcal{T}}}

\newcommand{\etauc}{\ensuremath{\eta^\mathrm{uni}_\mathrm{crit}}}
\newcommand{\etau}{\ensuremath{\eta^\mathrm{uni}}}

\newcommand{\FF}{\ensuremath{\mathbf{F}}}

\newcommand{\costu}{\ensuremath{C_\mathrm{uni}}}

\newcommand{\Cinf}{\ensuremath{C_\star}} 
\newcommand{\Csup}{\ensuremath{C^\star}} 
\newcommand{\oa}{\ensuremath{\Omega_\mathrm{a}}}
\newcommand{\oi}{\ensuremath{\Omega_\mathrm{i}}}

\newcommand{\I}{\ensuremath{\mathfrak{I}}}

\newcommand{\optProb}[3]{%
  \begin{equation}\label{#1}%
    \begin{cases}
      \textbf{Minimize: } & #2 \\
      \textbf{subject to: } & #3
    \end{cases}
  \end{equation}%
}
\newcommand{\optProbminmax}[3]{%
  \begin{equation}\label{#1}%
    \begin{cases}
      \textbf{Minimize: } & #2 \\
      \textbf{subject to: } & #3
    \end{cases}
   \quad  \quad\text{and}\quad \quad \quad 
      \begin{cases}
      \textbf{Maximize: } & #2 \\
      \textbf{subject to: } & #3
    \end{cases}
  \end{equation}%
}

\newcommand{\optProbTerm}[3]{%
  \begin{equation}\label{#1}%
    \begin{cases}
      \textbf{Minimize: } & #2 \\
      \textbf{subject to: } & #3
    \end{cases}
  \end{equation}%
 }
\newcommand{\optProbTerM}[3]{%
  \begin{equation}\label{#1}%
    \begin{cases}
      \textbf{Maximize: } & #2 \\
      \textbf{subject to: } & #3
    \end{cases}
  \end{equation}%
 }

\newsavebox{\largestimage}

\usepgfplotslibrary{colorbrewer}
\pgfplotsset{colormap/Paired-6}

\newcommand{\loss}{\ensuremath{\mathrm{L}}}

\newcommand{\lossup}{\ensuremath{\loss^\star}} 
\newcommand{\lossinf}{\ensuremath{\loss_\star}} 
\newcommand{\costa}{\ensuremath{C_\mathrm{aff}}} 
\newcommand{\costad}{\ensuremath{c_\mathrm{aff}}}

\newcommand{\param}{\ensuremath{\mathrm{Param}}}
\newcommand{\cplus}{c_{0}}
\newcommand{\maxcost}{c_{\max}}
\newcommand{\maxloss}{\ell_{\max}}

\newcommand{\etainf}{\eta_{\text{\tiny SO}}}
\newcommand{\etasup}{\eta_{\text{\tiny NE}}}

\newcommand{\cfi}{\ensuremath{\cf_\mathrm{inv}}}
\newcommand{\co}{\mathcal O}

\newcommand{\Id}{\ensuremath{\mathrm{Id}}}
\newcommand{\sq}{\ensuremath{\mathcal{A}}}

\usepackage[draft]{fixme}
\FXRegisterAuthor{pa}{apa}{PAZ}
\FXRegisterAuthor{jf}{ajf}{JFD}
\fxusetheme{colorsig}

\usepackage{booktabs}
\usepackage{multirow}
\usepackage{utfsym,bm} 
\usepackage{rotating}

\newcommand{\DeltaSIS}{\Delta_{\mathrm{SIS}}}
\newcommand{\DeltaOsc}{\Delta_{\mathrm{Osc}}}
\newcommand{\DeltaB}{\Delta_{\mathrm{Block}}}
\newcommand{\DeltaOrd}{\Delta_{\mathrm{Ord}}}
\newcommand{\KB}{K_{\mathrm{Block}}}
\newcommand{\KOrd}{K_{\mathrm{Ord}}}

\title{Optimal vaccination strategies for an heterogeneous  SIS model}

\begin{document}

\thanks{This work is partially supported by Labex B\'ezout reference ANR-10-LABX-58 and  SNF 200020-19699}



\begin{abstract}
  We  study  in  a  general
  mathematical  framework  the  optimal  allocation  of  vaccine  in  an
  heterogeneous population.   We cast
  the  problem of  optimal  vaccination as  a bi-objective  minimization
  problem  $\min(C(\eta),\loss(\eta))$,  where  $C$  and  $\loss$  stand
  respectively for  the cost  and the loss  incurred when  following the
  vaccination strategy~$\eta$,  where the function  $\eta(x)$ represents
  the proportion of non-vaccinated among individuals of feature~$x$.

  To measure  the loss,  we consider  either the  effective reproduction
  number, a classical quantity appearing in many models in epidemiology,
    or   the  overall  proportion  of  infected
  individuals after vaccination  in the maximal equilibrium,
  also called  the endemic state.  We  only make few assumptions  on the
  cost $C(\eta)$, which  cover in particular the uniform  cost, that is,
  the total number of vaccinated people.

  The  analysis of  the bi-objective  problem  is carried  in a  general
  framework, and we  check that it is  well posed for the  SIS model and
  has  Pareto  optima,   which  can  be  interpreted   as  the  ``best''
  vaccination strategies.   We provide  properties of  the corresponding
  Pareto frontier  given by  the outcomes  $( C(\eta),  \loss(\eta))$ of
  Pareto optimal strategies.

\end{abstract}

\maketitle

\section{Introduction}

\subsection{Motivation}

Targeted  vaccination  problems  have  mainly  been  studied  using  two
different points of view: eradication of a disease in heterogeneous
population, or removing vertices of a network for minimizing its
spectral radius.

\subsubsection{Optimal vaccination in heterogeneous population}

Increasing  the prevalence  of  immunity from  contagious  disease in  a
population limits the circulation of the infection among the individuals
who lack  immunity. This so-called  ``herd effect'' plays  a fundamental
role in epidemiology as it has had  a major impact in the eradication of
smallpox  and  rinderpest  or  the  near  eradication  of  poliomyelitis
\cite{HerdImmunityFine2011}. Targeted  vaccination strategies,  based on
the  heterogeneity of  the infection  spreading in  the population,  are
designed to  increase the  level of  immunity of  the population  with a
limited quantity of vaccine. These strategies rely on identifying groups
of individuals  that should be vaccinated  in priority in order  to slow
down or eradicate the disease.

In this article, we establish  a theoretical framework to study targeted
vaccination  strategies for  the deterministic  infinite-dimensional SIS
model    (with   S=Susceptible    and   I=Infectious)    introduced   in
\cite{delmas_infinite-dimensional_2020}, that  encompasses as particular
cases  the SIS  model  on graphs  or on  stochastic  block models.   The
simplicity  of the  SIS  model allows  us to  derive  strong results  on
optimal  vaccination strategies  and  existence of  Pareto optima  under
minimal assumptions for general non-homogeneous populations.  On the one
hand, using  this mathematical framework,  we can consider  in companion
papers a series of particular  cases and specific examples that complete
and illustrate  the present  work; see Section~\ref{sec:suite}  for more
details.  On the other hand, we expect the results obtained here for the
SIS model  to be generic,  in the  sense that behaviours  exhibited here
should  be  also  observed  in  more realistic  and  complex  models  in
epidemiology for non-homogeneous populations; in this direction, see for
example the discussion in \cite{ddz-Re}.

\subsubsection{Removing vertices of a network in an optimal way}

Mathematical epidemilogists usually consider a small number of
subpopulations, often with a dense structure of interaction, in the
sense that every subpopulation may directly infect all the others, see
Section~\ref{sec:lite} for further references. Other research
communities have looked into a similar problem for graphs.  Indeed,
given a large, possibly random, graph, with epidemic dynamics on it,
and supposing  we are able to suppress vertices by vaccinating,
one may ask for the best way to choose which vertices to remove.

The importance of the spectral radius of the network, that is, the
spectral radius of its adjacency matrix, has been quickly identified
as its value determines if the epidemic dies out quickly or survives
for a long time \cite{TheEffectOfNGaneshNone,
  CharacterizingRestre2006}. Since Van Mieghem \textit{et al.}\ proved
in~\cite{DecreasingTheVanMi2011} that the problem of minimizing
spectral radius of a graph by removing a given number of vertices is
NP-complete (and therefore unfeasible in practice), many computational
heuristics have been put forward to give approximate solutions; see
for example~\cite{ApproximationASaha2015} and references therein.  The
theoretical framework developed below gives the existence of Pareto optima
for the continuous relaxation of this problem.

\subsection{Herd immunity and targeted vaccination strategies}\label{subsec:problem}

We first recall  classical  results in mathematical epidemiology; we
refer to Keeling and Rohani's monograph~\cite{keeling_modeling_2008} for
an extensive  introduction to this  field, including details  on various
classical  models  such as  SIS,  SIR  and  SEIR (with  R=Recovered  and
E=Exposed).

In  an  homogeneous population,  the  basic  reproduction number  of  an
infection, denoted by~$R_0$, is defined as the number of secondary cases
one individual  generates on average  over the course of  its infectious
period, in an otherwise uninfected (susceptible) population. This number
plays  a fundamental  role in  epidemiology as  it provides  a scale  to
measure how difficult an infectious  disease is to control. Intuitively,
the  disease  should  die  out  if~$R_0<1$  and  invade  the  population
if~$R_0>1$. For many classical mathematical models of epidemiology, such
as  SIS   or  S(E)IR,   this  intuition  can   be  made   rigorous:  the
quantity~$R_0$ may be computed from the parameters of the model, and the
threshold phenomenon occurs.

Assuming~$R_0>1$ in an homogeneous population, suppose now that only a proportion~$\etau$
of the population can catch the disease, the rest being immunized. An infected
individual will now only generate~$\etau R_0$ new cases, since a proportion~$(1-\etau)$ of
previously successful infections will be prevented. Therefore, the new \emph{effective
reproduction number} is equal to~$R_e(\etau) = \etau R_0$. This fact led to the recognition
by Smith in 1970 \cite{smith_prospects_1970} and Dietz in 1975
\cite{dietz1975transmission} of a simple threshold theorem: the incidence of an infection
declines if the proportion of non-immune individuals is reduced below~$\etauc = 1/R_0$.
This effect is called \emph{herd immunity}.

It is of course unrealistic  to depict human populations as homogeneous,
and many generalizations of the homogeneous model have been studied, see
\cite[Chapter   3]{keeling_modeling_2008}  for   examples  and   further
references. For most  of these generalizations, it is  still possible to
define  a  meaningful  reproduction   number~$R_0$,  as  the  number  of
secondary cases  generated by  a \textit{typical}  infectious individual
when all  other individuals are uninfected  \cite{Diekmann1990}. After a
vaccination  campaign, let  the vaccination  strategy $\eta$  denote the
(non necessarily homogeneous)  proportion of the \textbf{non-vaccinated}
population, and let the effective reproduction number $R_e(\eta)$ denote
the corresponding reproduction number  of the non-vaccinated population.
The vaccination  strategy~$\eta$ is \emph{critical} if  $R_e(\eta) = 1$.
The possible choices  of $\eta$ naturally raises a question  that may be
expressed    as   the    following   informal    optimization   problem:
\optProb{eq:informal_optim1}{%
  \text{the quantity of vaccine to administrate}}{\text{herd immunity is
    reached,  that  is,~$R_e\leq 1$.}}   If  the  quantity of  available
vaccine    is   limited,    then    one   is    also   interested    in:
\optProb{eq:informal_optim2}{%
  \text{the effective reproduction number~$R_e$}}{\text{a given quantity
    of     available    vaccine.}}      Interestingly    enough,  when
$R_0>1$,    the
strategy~$\etauc$,   which   consists    in   delivering   the   vaccine
\emph{uniformly} to  the population,  without taking  inhomogeneity into
account,  leaves   a  proportion~$\etauc=   1/R_0$  of   the  population
unprotected,  and  is  therefore critical  since~$R_e(\etauc)  =1$.   In
particular     it     is     admissible     for     the     optimization
problem~\eqref{eq:informal_optim1}.

However, herd immunity may be achieved even if the proportion of unprotected people is
\emph{greater} than~$1/R_0$, by targeting certain group(s) within the population; see
Fig.~3.3 in \cite{keeling_modeling_2008}.

\medskip

The main goal of this paper is two-fold: formalize the optimization problems
\eqref{eq:informal_optim1} and~\eqref{eq:informal_optim2} for a particular infinite
dimensional SIS model, recasting them more generally as a bi-objective optimization
problem; and give existence and properties of solutions to this bi-objective problem. We
will also consider a closely related problem, where one wishes to minimize the size of the
epidemic rather than the reproduction number.

\subsection{Literature on targeted vaccination strategies}
\label{sec:lite}

Problems~\eqref{eq:informal_optim1} and~\eqref{eq:informal_optim2} have been examined in
depth for deterministic \emph{meta-population} models, that is, models in which an
heterogeneous population is stratified into a finite number of homogeneous subpopulations
(by age group, gender, \ldots). Such models are specified by choosing the sizes of the
subpopulations and quantifying the degree of interactions between them, in terms of
various mixing parameters. In this setting,~$R_0$ can often be identified as the spectral
radius of a \emph{next-generation matrix} whose coefficients depend on the subpopulation
sizes, and the mixing parameters. It turns out that the next generation matrices take
similar forms for many dynamics (SIS, SIR, SEIR,...); see the discussions in~\cite[Section
10]{hill-longini-2003} and~\cite[Section 2]{ddz-Re}.  Vaccination strategies are defined
as the levels at which each sub-population is immunized. After vaccination, the
next-generation matrix is changed and its new spectral radius corresponds to the effective
reproduction number~$R_e$.

Problem~\eqref{eq:informal_optim1} has been studied in this setting by Hill and Longini
\cite{hill-longini-2003}. These authors studied the geometric properties of the so-called
threshold hypersurface, that is, the vaccination allocations for which~$R_e = 1$. They also
computed the vaccination belonging to this surface with minimal cost for an Influenza A
model. Making structural assumptions on the mixing parameters, Poghotanyan, Feng, Glasser
and Hill derived  in \cite{poghotanyan_constrained_2018}  an analytical formula for the
solutions of Problem~\eqref{eq:informal_optim2}, for populations
divided in two groups. Many papers also
contain numerical studies of the optimization problems~\eqref{eq:informal_optim1}
and~\eqref{eq:informal_optim2} on real-world data using gradient techniques or similar
methods; see for example \cite{DistributionOfGoldst2010, feng_elaboration_2015,
TheMostEfficiDuijze2016, feng_evaluating_2017, IdentifyingOptZhao2019}.

\medskip

Finally, the effective reproduction number is not the only reasonable way of quantifying a
population's vulnerability to an infection. For an SIR infection for example, the
proportion of individuals that eventually catch (and recover from) the disease, often
referred to as the \emph{attack rate}, is broadly used. We refer
to~\cite{TheMostEfficiDuijze2016, DoseOptimalVaDuijze2018} for further discussion on this
topic.

\subsection{Main results}

\subsubsection{Costs and losses  in a SIS model}
The   differential  equations   governing  the   epidemic  dynamics   in
metapopulation~SIS  models were  developed by  Lajmanovich and  Yorke in
their     pioneer     paper~\cite{lajmanovich1976deterministic}.      In
\cite{delmas_infinite-dimensional_2020},   we   introduced   a   natural
generalization of this equation
to a possibly infinite space~$\Omega$,
where~$x  \in   \Omega$  represents   a  feature  and   the  probability
measure~$\mu(\mathrm{d} x)$  represents the  fraction of  the population
with  feature~$x$. We describe briefly and informally the key quantities
that we would like to minimize, and refer to
Section~\ref{sec:SIScostandloss} for details.

 A vaccination
strategy is represented by a function~$\eta:\Omega\to [0,1]$, where~$\eta(x)$ represents the
fraction of \textbf{non-vaccinated} individuals with feature~$x$.
In particular the strategy  $\eta=\un $ given by the
constant function equal to 1 corresponds to doing nothing;
similarly the constant strategy $\eta=\zero$ corresponds to 
vaccinating everybody.
The set of vaccination strategies
is denoted by $\DeltaSIS$. As mentioned before, to any vaccination  strategy we associate an epidemic dynamics on the
non-vaccinated individuals (that is, vaccinated individuals
are completely immunized and do not get infected nor infect others). For this effective dynamics,
we may compute two interesting  quantities.

The  first  is the  effective  reproduction  number $R_e(\eta)$  of  the
effective dynamics; a vaccination strategy should try to get this number
below  $1$.  If  the  effective reproduction  number $R_e(\eta)>1$,  the
generalization     of     the     SIS     dynamics     we     introduced
in~\cite{delmas_infinite-dimensional_2020},  just   like  its  classical
counterparts,  converges   to  a   non-trivial  equilibrium.    In  this
equilibrium a  certain number of people  are infected. The size  of this
equilibrium  infection under  the effective  dynamics defined  by $\eta$
will  be denoted  by $\I(\eta)$.   In  the SIS  model the  quantity~$\I$
appears as a natural analogue of the  attack rate in  the SIR model, and
is therefore a natural optimization objective.

Following~\cite{ddz-theory-topo}, we equip $\DeltaSIS$ with  the weak-* topology:
this choice makes $\DeltaSIS$ compact, and the loss functions $R_e$ and
$\I$ continuous.

To any vaccination strategy $\eta\in\DeltaSIS$ we also associate a
\emph{cost} $C(\eta)$, measuring the production and diffusion costs to
pay in order to deploy the strategy. Since the vaccination
strategy~$\eta$ represents the non-vaccinated population, the cost
function~$C$ should be non-increasing or decreasing (roughly
speaking~$\eta<\eta'$ implies $C(\eta)>C(\eta')$; see Definition
\ref{def:cont, monot}). We shall also assume that~$C$ is continuous
and that doing nothing costs nothing, that is,~$C(\un)=0$. A simple
and natural choice is the uniform cost~$\costu$ given by the overall
proportion of vaccinated individuals, see Remark~\ref{rem:costa-costu}
for comments on other examples of cost functions.

\subsubsection{Optimizing the protection of the population}

Our problem may now be seen as a bi-objective minimization problem: we wish to minimize
both the loss~$\loss(\eta)$ and the cost~$C(\eta)$, subject to $\eta \in \DeltaSIS$, with the
loss function~$\loss$ being either~$R_e$ or $\I$. Following classical terminology for
multi-objective optimisation problems \cite{NonlinearMultiMietti1998}, we call a
strategy~$\eta_\star$ \emph{Pareto optimal} if no other strategy is strictly better:
\[
  C(\eta)< C(\eta_\star) \implies \loss(\eta) > \loss(\eta_\star)
  \quad\text{and}\quad
  \loss(\eta)< \loss(\eta_\star) \implies C(\eta) > C(\eta_\star).
\]
The set of Pareto optimal strategies, which might be
  empty \emph{a priori}, will be denoted by~$\cp$.  We define the
\emph{Pareto frontier} as the set of Pareto optimal outcomes:
\[
  \mathcal{F} =
  \{ (C(\eta_\star),\loss(\eta_\star))\in \R^2 \, \colon \, \eta_\star \in
  \cp \}.
\]

For any strategy~$\eta$, the cost and loss of~$\eta$ vary between the following bounds:
\[
  \begin{alignedat}{3}
  0 &= C(\un) \leq C(\eta) &&\leq C(\zero)& &= \maxcost = \text{cost of vaccinating the whole population},\\
  0 &=\loss(\zero)\leq \loss(\eta) &&\leq \loss(\un)& &= \maxloss= \text{loss incurred in the absence of
 vaccination}.
  \end{alignedat}
  \]
Let~$\lossinf$ be the \emph{optimal loss} function and~$\Cinf$ the \emph{optimal} cost
function defined by:
\begin{align*}
  \lossinf(c) &=
  \inf \, \set{ \loss(\eta) \, \colon \, \eta \in \DeltaSIS, \, C(\eta) \leq c } \quad
  \text{for $c\in [0,\maxcost]$}, \\
  \Cinf(\ell) &=
  \inf \, \set{ C(\eta) \, \colon \, \eta \in \DeltaSIS,\, \loss(\eta) \leq \ell } \quad
  \text{for $\ell \in [0,\maxloss]$}.
\end{align*}

Our main result, illustrated in Fig.~\ref{fig:pareto_frontier} below,
states that the Pareto frontier is non empty and has a continuous
parametrization for both losses~$\loss=R_e$ and~$\loss=\I$.
  More
precisely, using results developed  in a more general
framework, see Section~\ref{sec:P-AntiP},  we get the following
result, see  Section~\ref{sec:summary} and
Tab.~\ref{tab:not-PAP} therein for a more complete picture.

\begin{theorem}[Properties of the Pareto frontier]
  \label{thm:1}
 Consider  the  SIS model   with the loss
 $\loss\in \{R_e, \I\}$ and the
 uniform cost. Under technical integrability assumptions, see Assumption~\ref{hyp:k-g},
  \begin{enumerate}[(i)]
   \item The
     function~$\Cinf$ is continuous and decreasing on~$[0, \maxloss]$,
     \item The
  function~$\lossinf$  is   continuous  on~$[0,   \maxcost]$  decreasing
  on~$[0,  \Cinf(0)]$ and  zero on~$[\Cinf(0),\maxcost]$.
  \item The Pareto frontier is compact, connected and:
  \[
    \mathcal{F} = \{(c,\lossinf(c)) \, \colon \, c \in [0,\Cinf(0)]\} =
    \{(\Cinf(\ell), \ell) \, \colon \, \ell \in [0,\maxloss]\} .
  \]
  \item The set of Pareto optima $\cp$ is compact
    in~$\DeltaSIS$.
  \end{enumerate}
\end{theorem}

In this  introduction we  focus, for readability's  sake, on  a concrete
setting.  Let us note that the  result stated above is a particular case
of  the  more general  statements  proved  below,  that we  now  briefly
discuss.  The  set of vaccinations, denoted  by $\Delta$, may be  a more
general      partially      ordered     topological      space      than
$\DeltaSIS =  \{\eta: \zero\leq \eta\leq  \un \}$; in particular  we may
consider  strict subsets  of~$\DeltaSIS$,  to model  constraints on  the
strategies  that  are  possible  to  deploy, see the examples in
Section~\ref{sec:vaccination-strat}
and Remark~\ref{rem:autre-DSIS}.   We  are  able  to  tackle
different,  non-uniform,  cost  functions,   and  other  loss  functions
provided they have enough regularity.

This  generality  allows  us  to  give also  results  on  the  ``worst''
strategies,  which  we  call  anti-Pareto optimal,  as  opposed  to  the
``best'', Pareto-optimal ones.  Surprisingly,  proving properties of the
anti-Pareto frontier sometimes necessitates stronger assumptions than in
the  Pareto   case.  For   example,  under   some  \emph{irreducibility}
assumption on  the kernel developed  in Section~\ref{sec:prop-L-atomic},
we establish the connectedness of the anti-Pareto frontier when the loss
is given  by $\loss=R_e$ (see also~\cite[Section~5]{ddz-cordon}  in this
direction); and a slightly different behavior  when the loss is given by
$\loss=\I$.

We  also  establish  that  the   set  of  outcomes  or  feasible  region
$\FF=\{  (C(\eta),\loss(\eta)), \,  \eta\in  \Delta\}$ has  no holes  in
Proposition~\ref{prop:trou};  and that  the  Pareto  frontier is  convex
if~$C$  and~$\loss$ are  convex  in Proposition~\ref{prop:cvex}.   Using
Proposition~\ref{prop:F-stab}  and results  from~\cite{ddz-theory-topo},
we can also  get the stability of  the Pareto frontier and  the set of
Pareto optima when the parameters of the SIS model vary.

\begin{remark}[Optimal critical  strategies do not depend on the
  loss]\label{rem:CR1=CI0}%
  

  Recall  a  vaccination  strategy   $\eta\in  \Delta$  is  critical  if
  $R_e(\eta)=1$.   Consider  the  uniform  cost  $C=\costu$,  and  write
  $\cp_\loss$ for the Pareto optima  associated to the loss $\loss$.  Of
  course,  there  is  no  reason  for $\cp_\I$  and  $\cp_{R_e}$  to  be
  comparable.  Nevertheless,  we have  the following result  on critical
  Pareto optima:
  \begin{equation}
    \label{eq:PJ-PRe}
    \eta\in \cp_{R_e}\text{ and } R_e(\eta)=1 \quad \Longleftrightarrow
    \quad \eta\in \cp_{\I}\text{ and } \I(\eta)=0.
  \end{equation}
  Indeed,  in \cite{delmas_infinite-dimensional_2020},  we proved  that,
  for all~$\eta\in\Delta$, the  equilibrium infection size~$\I(\eta)$ is
  non zero if and only  if~$R_e(\eta)>1$.  This implies that~$\cp_\I$ is
  a  subset   of~$\{  \eta\in  \Delta\,  \colon\,   R_e(\eta)\geq  1\}$,
  and~\eqref{eq:PJ-PRe} follows thanks to~Theorem~\ref{thm:1}.
  \end{remark}

\subsubsection{An illustrative example: the multipartite graph}
\label{sec:multipartite}

Let us illustrate some of our results on an example, which is 
discussed in details in the companion paper \cite{ddz-reg}.

  Graphs whose vertices  can be colored with~$\ell$ colors,  so that the
  endpoints   of  every   edge  are   colored  differently,   are  known
  as~$\ell$-partite graphs. In a biological setting, this corresponds to
  a population of~$\ell$  groups, such that individuals in  a group only
  contaminate individuals  of the other  groups.
     (For example, sexually
  transmitted  disease can be roughly approximated as an epidemic on a
  bipartite graph, see~\cite{mst}.)

Let us generalize and assume there is an infinity of
groups,~$\ell=\infty$, of respective size~$(2^{-n}, \, n\in \N^*)$,
and that there is a constant inter-group positive rate and that there
is no intra-group contamination; see Fig.~\ref{fig:pareto_frontier}. 
  
  \medskip
  
  Consider the loss~$\loss = R_e$ and the uniform cost~$C=\costu$ giving
  the  overall proportion  of vaccinated  individuals.  Building  on the
  results   of   \cite{esser_spectrum_1980,stevanovi},   we   prove   in
  \cite{ddz-reg} that vaccinating individuals with the highest number of
  connections (that is individuals in  the smallest groups) and stopping
  when all the  groups but the first one, are  fully vaccinated, are the
  Pareto  optimal  vaccinations.   In  particular,  since  the  cost  of
  vaccinating the  whole population but  the first group (whose  size is
  $1/2$ in  a total population  of size $1$)  has cost $1/2$,  we deduce
  that $\Cinf(0)=1/2$.

  In Fig~\ref{fig:pareto_frontier}, the corresponding Pareto
  frontier (\textit{i.e.}, the outcome of the ``best'' vaccination
  strategies) is drawn as the solid red line; the blue-colored zone
  corresponds to the feasible region, that is, all the possible values
  of~$(C(\eta), R_e(\eta))$ where~$\eta$ ranges over~$\DeltaSIS$; the
  dotted line corresponds to the outcome of the uniform vaccination
  strategy, that is $(C(\eta), R_e(\eta))=(c, (1-c) R_0)$ where~$c$
  ranges over~$[0, 1]$; and the red dashed curve corresponds to the
  anti-Pareto frontier (\textit{i.e.}, the outcome of the ``worst''
  vaccination strategies), which for this model correspond to the
  uniform vaccination of the nodes with the updated lower degree; see
  \cite{ddz-reg}.  The set of optimal Pareto vaccination strategies
  given by vaccinating the individuals with the highest number of
  connections gives a complete parametrization of the Pareto frontier.

  \begin{figure}
    \centering
    \begin{subfigure}[T]{.55\textwidth}
      \includegraphics{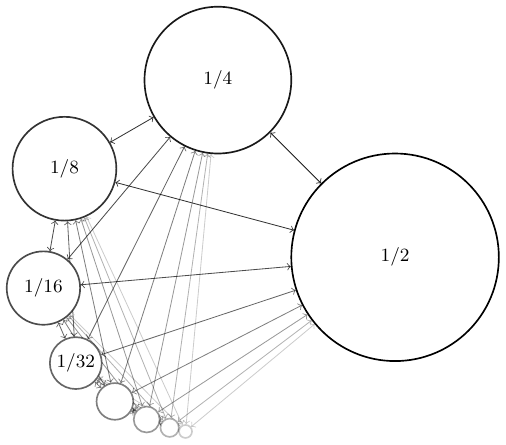}%
\caption{Partial  multipartite graph.}
\end{subfigure}%
\begin{subfigure}[T]{.45\textwidth} \centering
  \includestandalone[page=2]{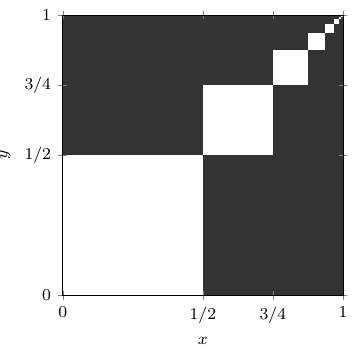}
  \caption{The Pareto frontier in solid red
    line compared  to the cost and  loss of the uniform  vaccinations in
    dotted line and the worst vaccination strategy in red dashed line.}
\end{subfigure}%
 \caption{Example
    of optimization with~$\loss = R_e$ for a multipartite graph with no
    intra-population transmission.}
 \label{fig:pareto_frontier}
\end{figure}

\subsection{On the companion papers}\label{sec:suite}

Let us discuss briefly the results that can be found in the companion
papers~\cite{ddz-theory-topo,ddz-theory-couplage}, and the related
papers \cite{ddz-Re,ddz-cordon,ddz-reg,ddz-hit,ddz-mono} which are
given in the framework developed here, and illustrate and complete in
various directions the present mathematical foundational work.

In the first  companion paper \cite{ddz-theory-topo}, we  prove that the
loss  functions  we  consider  for  the SIS  model  indeed  satisfy  the
continuity  assumptions  that  we  use here  for  proving  existence  of
optima.
  In  the second one, \cite{ddz-theory-couplage},  we describe how
two seemingly different feature spaces  may in fact describe essentially
the  same model;
in particular  we show  how discrete  models like  the
stochastic block models may be encoded in a continuous model, giving the
same Pareto and anti-Pareto frontiers.

\medskip

These core results are then used and extended in various directions.
The discussion of vaccination  control of gonorrhea
in~\cite[Section~4.5]{hethcote}  suggests  that  it  may  be  better  to
prioritize  the  vaccination of  people  that  have already  caught  the
disease:  this  leads  us  to  consider  in  \cite{ddz-hit}  a  critical
vaccination strategy guided by the equilibrium state. Note that it seems
difficult to give conditions for  this vaccination strategy to be Pareto
or anti-Pareto optimal.

In  \cite{ddz-Re}, motivated  by  a conjecture  formulated  by Hill  and
Longini                in                finite                dimension
(see~\cite[Conjecture~8.1]{hill-longini-2003}),   we   investigate   the
convexity and concavity of the effective reproduction function~$R_e$. In
\cite{ddz-cordon},   we  prove   for  the   loss  $\loss=R_e$   that:  a
disconnecting strategy is  better than the worst,  \textit{i.e.}, is not
anti-Pareto optimal; and give a graph interpretation of $\Cinf(0)$ which
is the minimal cost to ensure  that no infection occurs. We also provide
in  this setting  further results  on the  anti-Pareto frontier  using a
decomposition of the kernel on its irreducible components.

In \cite{ddz-reg},  we study the  Pareto and anti-Pareto frontier  for a
number of  examples.  In some  of them,  like the multipartite  graph of
Section~\ref{sec:multipartite},  the  optimal  vaccinations  target  the
individuals with  the highest number  of contacts. When  the individuals
have the same number of contacts, this heuristic cannot be used, and the
situations  may be  extremely varied:  for example,  we give  models for
which the uniform vaccination is Pareto optimal, or anti-Pareto optimal,
or not optimal for either problem.  We also provide an example where the
set of Pareto optima has a countable number of connected components (and
is  thus not  connected). This  implies  in particular  that the  greedy
algorithm is not optimal in this case.

\subsection{Structure of the paper}

We present the multi-objective optimization problem in
Section~\ref{sec:P-AntiP} in an~abstract setting, under general
conditions on the loss and cost functions and prove the first results
on the bi-objective problem and its links with the single-objective
problems obtained when fixing the value of either the cost or the
loss, and optimizing the other. This is completed in
Section~\ref{sec:diversCL} with miscellaneous properties of the Pareto
frontier.  Section~\ref{sec:k-SIS-model} is dedicated to the proper
presentation of the SIS setting, and the consequences of the general
results in this framework.

\section{Pareto and anti-Pareto frontiers}\label{sec:P-AntiP}

\subsection{The setting}\label{sec:cost-loss}

\subsubsection{Vaccination strategies}
\label{sec:vaccination-strat}

In the  SIS model  from~\cite{delmas_infinite-dimensional_2020} recalled
in Section~\ref{sec:k-SIS-model}, vaccination  strategies are encoded by
functions  from   the  trait  space   $\Omega$  to  $[0,1]$:   we  equip
$E =  L^\infty(\Omega)$ with  the weak-* topology  and with  the natural
partial     order $\leq $    on      real-valued     functions,     and     let
  $\Delta  = \DeltaSIS =  \{ \eta\in  E\, \colon\,   \zero\leq  \eta  \leq \un\}$,
with $\zero$  (resp.  $\un$)
  the constant  function equal  to zero  (resp. one).    We
generalize this  setting in this  section and  the next one,  keeping in
mind the SIS case as a guiding example.

  \medskip

  Let $(E, \co)$ be a Hausdorff topological vector space, and denote by
  $\zero$ its  zero element.
  Let $K\subset E$ be a salient pointed convex cone, that is,
  $\lambda K \subset K$ for all $\lambda\geq 0$, $K+K\subset K$ and
  $K \cap (-K)=\{\zero\}$.  Notice we don't assume that $K$ is
  closed. The cone $K$ defines a partial order $\preceq $ on $E$: for
  $\eta, \eta'\in E$, we set $\eta'\preceq \eta$ if $\eta-\eta'\in K$.
  We also write $\eta'\prec \eta$ if $\eta'\preceq \eta$ and
  $\eta'\neq \eta$.  For $\Delta\subset E$, we say that $\eta$ is
  \emph{the maximum} of $\Delta$ and write $\eta=\max \Delta$ if
  $\eta\in \Delta$ and $\eta'\preceq \eta$ holds for all
  $\eta'\in \Delta$; we also say that $\eta$ is \emph{the minimum} of
  $\Delta$ and write $\eta=\min \Delta$ if $\eta=-\max (-\Delta)$.

  From now  on, we shall  consider a subset  $\Delta \subset E$  with an
  element $\un \in \Delta$ such that:
  \begin{equation}
   \label{eq:D-prop}
   \Delta\quad\text{is convex, compact,}\quad   \zero=\min  \Delta
   \quad\text{and}\quad 
   \un=\max  \Delta.
  \end{equation}
  In particular, the set $\Delta$ is closed (as a compact subset of a
  Hausdorff space), and the latter two conditions of~\eqref{eq:D-prop}
  are equivalent to $\zero, \un \in \Delta$ and
  $\Delta \subset K \cap (\un -K)$.  Let us stress that writing
  $\un=\max \Delta$ implies that all elements of $\Delta$ are
  comparable with (and thus less than) $\un$.

  Naturally in the SIS model, the set  $\DeltaSIS $
  defined above 
  satisfies~\eqref{eq:D-prop}.
We now discuss three additional choices for $\Delta$,
  which  are all subsets of $\DeltaSIS$.
\begin{description}

\item[Almost uniform vaccinations]
  Let $\DeltaOsc$  denote the set of strategies $\eta\in\DeltaSIS$ that are ``almost uniform''
  in the sense that the oscillation
  $\sup  \eta -  \inf \eta  \leq    \delta$
  for some  given $\delta\in (0, 1)$, where  the extrema are
  understood as essential extrema. Restricting strategies to $\DeltaOsc$
  ensure that no subset of the population is too favored. 
\item[Blockwise strategies] Let $(\Omega_i)_{i\in I}$ be a measurable
  partition of $\Omega$ into sets of positive measure,
  and consider the set of corresponding
  blockwise constant vaccination strategies
  $\DeltaB=\{\eta\in \DeltaSIS\, \colon\, \eta \text{ constant on each }
  \Omega_i\}$.
  Restricting strategies to $\DeltaB$ is
  natural if one considers that public policies will typically focus
  on a finite set of identifiable subgroups and treat each subgroup
  uniformly.
    \item[Ordered blockwise  strategies] We elaborate on the previous
      example by adding another constraint on the possible strategies. 
      Assume that the index set~$I$ is totally ordered.  Consider the
      subset of the corresponding block wise constant vaccination
      strategies
      $\DeltaOrd=\{\eta\in \DeltaB\, \colon\, \eta_i\leq \eta_j\text{
        for } i\leq j\}$, where $\eta_i$ is the (constant) value of
      $\eta\in \DeltaOrd$ on $\Omega_i$. This corresponds to situations where some
      categories must be at least as vaccinated as others. 

    \end{description}
    The three sets $\DeltaOsc$, $\DeltaB$ and~$\DeltaOrd$ all satisfy~\eqref{eq:D-prop}.

\begin{remark}[Other properties of the various sets of vaccination strategies]
  \label{rem:prop-of-Delta}
Let us briefly discuss three additional properties that the set $\Delta$ may enjoy. 
\begin{enumerate}[(i)]
\item  $\Delta$ is \emph{saturated}  if whenever $\eta_1\preceq \eta_2\preceq \eta_3$,
  with $\eta_1$ and $\eta_3$  in $\Delta$, then $\eta_2\in \Delta$;
\item $\Delta$ \emph{generates} $K$ if  $K= \R_+ \Delta$;
\item $\Delta$ is \emph{symmetric w.r.t. $\un/2$} if  $\Delta=\un -\Delta$.
\end{enumerate}
The set $\DeltaSIS$ clearly satisfies these properties. 
Note that if \eqref{eq:D-prop} is satisfied, then $\Delta$ is saturated if
 and only if $\Delta = K\cap(\un - K)$, and is then necessarily symmetric.

\medskip

    The set  $\DeltaOsc$
  satisfies~\eqref{eq:D-prop}, it  generates  $K$  and
    $\DeltaOsc=\un -\DeltaOsc$.  However, it  is not saturated.
      
    \smallskip
    
    The set $\DeltaB=\un  -\DeltaB$ is symmetric.  However,  it is not
    saturated nor generates  $K$, unless the $\Omega_i$  are atoms.  One
    may be  tempted to consider an  order associated to a  smaller cone,
    say $\KB$, to  recover these properties: this also  enlarges the set
    of monotone functions, which may be  interesting in the light of the
    monotonicity   assumptions  below   on   the  cost   and  loss   (if
    $\KB\subset K$, monotone  functions with respect to  the order given
    by  $\KB$,  see  Definition~\ref{def:cont, monot}  below,  are  also
    monotone with respect to the order  given by $K$, the opposite being
    false in general).  To this end, let $\KB=\R_+ \DeltaB$ be the cone
    generated by  $\DeltaB$, and observe  then that, w.r.t.\  the order
    generated  by $\KB$,  $\un$  is  still larger  than  any element  of
    $\DeltaB$  (in  particular  $\DeltaB$  satisfies~\eqref{eq:D-prop}
    with  the order  generated by  $\KB$), and  that $\DeltaB$  is then
    saturated.

      \smallskip
      
      However,  this  trick  of  restricting  the  cone  is  not  always
      possible.   Indeed,  consider  the  set~$\DeltaOrd$  (that  is  not
      saturated and does not generate $K$) and assume there are at least
      two blocks.   Since $\DeltaOrd\cap  (\un -\DeltaOrd)$ is  reduced to
      the  constant vaccination  strategies, it  is a  strict subset  of
      $\DeltaOrd$,  that  is,  $\DeltaOrd\neq \un  -\DeltaOrd$.   Therefore
      $\DeltaOrd$  cannot  be  saturated  with  respect  to  a  cone  for
      which~\eqref{eq:D-prop} is satified.  For  example, if one defines
      $\KOrd=\R_+\DeltaOrd$,  then~\eqref{eq:D-prop}  fails, because  there
      are   $\eta\in\DeltaOrd$  such   that   $\un-\eta\notin  \KOrd$,   so
      $\eta\nprec \un$ for the order defined by~$\KOrd$.
\end{remark}

\subsubsection{Costs and losses}

We let $E^*$ denote the  topological dual  of $E$,  that is, the set
of continuous real valued linear functions defined on $E$, with the dual
pairing   $\langle   \cdot,   \cdot   \rangle$,  and   the   dual   cone
$K^*=\{x^*\in   E^*\,  \colon\,   \langle   x^*,   \eta  \rangle\geq   0
\text{ for all $ \eta\in K$}\}$.

\begin{definition}\label{def:cont, monot}
  We say that a real-valued function $H$ defined on $\Delta$ is:
  \begin{itemize}
    \item \textbf{\emph{Non-decreasing}}: if for any $\eta_1, \eta_2\in \Delta$ such that
      $\eta_1\preceq \eta_2$, we have $H(\eta_1)\leq H(\eta_2)$.
    \item       \textbf{\emph{Decreasing}}:       if       for       any
      $\eta_1, \eta_2\in \Delta$ such that $\eta_1\prec \eta_2$, we have
      $H(\eta_1)> H(\eta_2)$.
  \item \textbf{\emph{Sub-homogeneous}}: if $H(\lambda \eta)\leq \lambda H(\eta)$ for
      all $\eta\in \Delta$ and $\lambda\in [0, 1]$.
    \item \textbf{\emph{Affine}}: if there exists $c\in\mathbb{R}$ and $x^*\in E^*$ such that
  \(  H(\eta)= c - \langle x^*,\eta \rangle\). 
  \end{itemize}
\end{definition}

The definition  of non-increasing  function and increasing  function are
similar. We shall  study a bi-objective optimization problem  for a cost
function        $C$       and        loss       function        $\loss$,
see~\eqref{eq:multi-objective-minmax}  below,  under various  regularity
conditions. In particular the next assumption will always hold.

\begin{hyp}[On the cost function and loss function]\label{hyp:loss+cost}
  The loss function $\loss: \Delta \rightarrow \R$ is non-decreasing and continuous with
  $\loss(\zero)=0$. The cost function $C: \Delta\rightarrow
  \R$ is non-increasing and continuous with $C(\un) = 0$.
    The loss and
    cost functions are not degenerate, that is:
  \[
  \maxloss:=\max_\Delta \, \loss>0\quad\text{and}\quad \maxcost:=\max_\Delta \, C>0.
  \]
\end{hyp}
In particular, the loss
and the cost functions are non-negative and non-zero. Notice that
$\maxloss=\loss(\un)$ and $\maxcost=C(\zero)$.

  \begin{remark}[Affine cost functions]
    \label{rem:on-C-aff}
    Let $x^*\in K^*$  be such that $\langle x^*, \un  \rangle\neq 0$.
    The restriction to~$\Delta$ of the affine function defined    on   $E$   by
    $\eta  \mapsto  C(\eta)=\langle x^*,  \un  -\eta  \rangle$
    clearly satifies
        Assumption~\ref{hyp:loss+cost}.

        One may wonder conversely if any affine cost function satisfying
        Assumption~\ref{hyp:loss+cost}     may     be     written     as
        $\langle x^*,  \un - \eta \rangle$  for an $x^*\in K^*$.   It is
        indeed the case  in some situations, for example  if $\Delta$ is
        saturated and generates $K$, and  thus in particular for the SIS
        model with  $\DeltaSIS$.  This also holds  trivially if $\Omega$
        is finite  and it  also holds for  all the  examples $\Delta_i$,
        $i=1,2,3$    given     in    Section~\ref{sec:vaccination-strat}
        concerning  the SIS  model.  The  general case  boils down  to a
        question of continuous extension  of positive linear functionals
        and is not clear in the full generality of this section.
      \end{remark}

\subsection{Optimization problems}
\label{sec:opt-pb}

We will consider the multi-objective minimization and maximization problems:
\optProbminmax{eq:multi-objective-minmax}{(C(\eta),\loss(\eta))}{\eta
  \in \Delta}

Multi-objective  problems are  in a  sense ill-defined  because in  most
cases, it is impossible to find  a single solution that would be optimal
to all objectives simultaneously. Hence, we recall the concept of Pareto
optimality.  Since   the  minimization   problem  is  crucial   for  our
application to  vaccination, we shall  define Pareto optimality  for the
bi-objective minimization problem. A strategy $\eta_\star \in \Delta$ is
said  to  be  \emph{Pareto  optimal} for  the  minimization  problem  in
\eqref{eq:multi-objective-minmax}  if any  improvement of  one objective
leads to a deterioration of the other, for $\eta\in \Delta$:
\begin{equation}\label{eq:defParetoOptimal}
  C(\eta)<C(\eta_\star) \implies \loss(\eta) >\loss(\eta_\star)
  \quad\text{and}\quad
  \loss(\eta)<\loss(\eta_\star) \implies C(\eta) > C(\eta_\star).
\end{equation}

The set of
Pareto optimal strategies will be denoted  by~$\cp$,  and the Pareto frontier is
defined as the set of Pareto optimal outcomes:
\[
  \mathcal{F} = \{ (C(\eta),\loss(\eta))\,\colon\, \eta \text{ Pareto optimal}\}.
\]
Similarly,  a  strategy   $\eta^\star\in  \Delta$  is  \emph{anti-Pareto
  optimal} if  it is  Pareto optimal  for the  bi-objective maximization
problem  in  \eqref{eq:multi-objective-minmax}.   Intuitively,  for  the
vaccination  application, the  ``best'' vaccination  strategies are  the
Pareto  optima   and  the  ``worst''  vaccination   strategies  are  the
anti-Pareto optima.  We  denote similarly by~$\cpa$
the    set     of    anti-Pareto    optimal    strategies,     and    by
$\AF$ its frontier:
\[
  \AF = \{ (C(\eta), \loss(\eta))\,\colon\, \eta
  \text{ anti-Pareto optimal}\}.
\]
Finally, we define the \emph{feasible region} as the set of  all possible outcomes:
\[
  \FF =\{ (C(\eta),\loss(\eta)), \,\eta\in \Delta\}.
\]

We now consider the classical, single-objective minimization problems
related to the ``best''  strategies, with
a fixed loss~$\ell \in [0, \maxloss]$ or a fixed cost~$c\in [0, \maxcost]$:
\optProbTerm{eq:Prob1}{\loss(\eta)}{\eta\in\Delta, \,
  C(\eta)\leq c,}
 as well as
\optProbTerm{eq:Prob2}{C(\eta)}{\eta\in\Delta, \,
  \loss(\eta)\leq \ell.}
We denote the values of Problems~\eqref{eq:Prob1} and
\eqref{eq:Prob2} by:
\begin{align*}
  \lossinf(c)
&=\inf\{ \loss(\eta)\, \colon\, \eta\in \Delta \text{ and }
C(\eta) \leq c\} \quad\text{for $c\in [0, \maxcost]$},\\
\Cinf(\ell)
&=\inf\{ C(\eta)\, \colon\, \eta\in \Delta \text{ and } \loss(\eta) \leq
\ell\}\quad\text{for $\ell\in [0, \maxloss]$}.
\end{align*}

Similarly,  the maximization problem related to the ``worst''
 strategies
for a fixed loss~$\ell \in [0, \maxloss]$ or a fixed cost~$c\in [0,
\maxcost]$
are defined by: 
\optProbTerM{eq:Prob1**}{\loss(\eta)}{\eta\in\Delta, \,
C(\eta)\geq c,}
 as well as
\optProbTerM{eq:Prob2**}{C(\eta)}{\eta\in\Delta, \,
  \loss(\eta)\geq \ell.}

We denote the values of Problems~\eqref{eq:Prob1**} and
\eqref{eq:Prob2**} by:
\begin{align*}
  \lossup(c)
&=\sup\{ \loss(\eta)\, \colon\, \eta\in \Delta \text{ and }
C(\eta) \geq c\} \quad\text{for $c\in [0, \maxcost]$},\\
\Csup(\ell)
&=\sup\{ C(\eta)\, \colon\, \eta\in \Delta \text{ and } \loss(\eta) \geq
\ell\}\quad\text{for $\ell\in [0, \maxloss]$}.
\end{align*}

Under Assumption \ref{hyp:loss+cost}, as the loss and the cost functions
are  continuous  on   the  compact  set~$\Delta$,  the   infima  in  the
definitions of  the value functions  $\Cinf$ and $\lossinf$  are minima;
and the  suprema in the  definition of  the value functions  $\Csup$ and
$\lossup$  are  maxima.

\medskip

See Fig.~\ref{fig:typical_frontier} for  a typical  representation of
the possible aspects  of the feasible region~$\FF$ (in  light blue), the
value  functions and  the  Pareto and  anti-Pareto  frontiers under  the
general  Assumption~\ref{hyp:loss+cost}, and  the  connected Pareto  and
anti-Pareto  frontiers under  further regularity  on the  cost and  loss
functions  (see   Assumption~\ref{hyp:cost}-\ref{hyp:loss**}  below)  in
Fig.~\ref{fig:reg_frontier}.  In  Fig.~\ref{fig:pareto_frontier}, we
have plotted  in solid red  line the Pareto  frontier and in  dashed red
line the anti-Pareto frontier from Section~\ref{sec:multipartite}.

\subsubsection{Outline of the section}

It turns out that the anti-Pareto  optimization problem can be recast as
a Pareto optimization problem by  changing signs and exchanging the cost
and  loss functions.   In order  to  make use  of this  property in  our
application for vaccination  in SIS models, we study  the Pareto problem
under assumptions on  the loss $L$ that are general  enough to cover the
effective  reproduction  function $R_e$  and  the  fraction of  infected
individuals at equilibrium $\I$, and  assumptions on the cost that cover
the uniform or affine cost functions and $-\loss$.

\medskip

The main  result of this  section states that  all the solutions  of the
optimization  Problems~\eqref{eq:Prob1}  or~\eqref{eq:Prob2} are  Pareto
optimal,   and   gives   a    description   of   the   Pareto   frontier
$\mathcal{F}$  as  a  graph  in  Section  \ref{sec:Pareto-F},  and
similarly for  the anti-Pareto frontier in  Section \ref{sec:Pareto-AF}.
Surprisingly, those  two problems are  not completely symmetric  for the
loss   functions  $R_e$   and  $\I$   for  SIS   models  considered   in
Section~\ref{sec:prop-C+L},     see     Lemma     \ref{prop:L**}
,  where one  uses some  irreducibility condition  on the
kernels to study the anti-Pareto frontier.

\subsection{On the Pareto frontier}\label{sec:Pareto-F}

We first check that Problems \eqref{eq:Prob1} and \eqref{eq:Prob2} have solutions.

\begin{proposition}[Optimal solutions for fixed cost or fixed
  loss]\label{prop:main-result}
  Suppose  Assumption~\ref{hyp:loss+cost} holds. For any cost $c\in[0,\maxcost]$,
  there exists a minimizer of the loss under the cost constraint $C(\cdot)\leq c$, that
  is, a solution to Problem~\eqref{eq:Prob1}. Similarly, for any loss
  $\ell\in[0,\maxloss]$, there exists a minimizer of the cost under the loss constraint
  $\loss(\cdot) \leq \ell$, that is a solution to Problem~\eqref{eq:Prob2}.
\end{proposition}

\begin{proof}
  Let $c \in [0,\maxcost]$.  Since $C(\un)=0$, the set $\{\eta\in \Delta \,\colon\, C(\eta)\leq c\}$ is
  non-empty. It is also compact as $C$ is continuous on
  the compact set $\Delta$. Therefore, since the loss function
  $\loss$ is continuous, we get that $\loss$ restricted to this
  compact set reaches its minimum:  Problem~\eqref{eq:Prob1} has a solution. The
  proof is similar for the existence of a solution to Problem~\eqref{eq:Prob2}.
\end{proof}

  \begin{proposition}[Elementary properties of
    $\lossinf,\lossup,\Cinf,\Csup$]
    \label{prop:elementary}
    If   Assumption~\ref{hyp:loss+cost}   holds,  then   the   functions
    $\lossinf$  and $\Cinf$  are right-continuous,  while $\lossup$  and
    $\Csup$   are    left-continuous   ;   the   four    functions   are
    non-increasing.     We     also     have:
    \[
      \lossinf(\maxcost)=0,\quad
    \lossup(0)=\maxloss,\quad \Cinf(\maxloss)=0 \quad\text{and}\quad
    \Csup(0)=\maxcost.
    \]
  \end{proposition}
  \begin{proof}
    We only consider $\lossinf$, the three other cases being similar.
    By definition, $\lossinf$ is non-increasing.
    It remains to check that it is right-continuous.
    Let $c_n$ be a decreasing sequence converging  to $c$, and
    let $\ell_n = \lossinf(c_n)\leq  \lossinf(c)$. By monotonicity, the sequence $\ell_n$
    is non-decreasing and converges to a limit $\ell_\infty \leq \lossinf(c)$. 
    To prove the reverse inequality, recalling that $\FF$
    denotes the set of all possible outcomes, we write:
    \[
      \lossinf(c) = \inf\{ L(\eta)\, \colon\,  \eta \in \Delta, C(\eta) \leq c\} 
                  = \inf\{ \ell\, \colon\,  (c',\ell) \in \FF, c'\leq c\}.
    \]
    Since $\FF$ is compact (as the  image of the compact set $\Delta$ by
    the continuous  map $(C,\loss)$),  and the  condition $c'\leq  c$ is
    closed,  the  infimum  is  a minimum.   By  definition  there  exist
    $(c'_n,\ell_n)$    such    that   $(c'_n,\ell_n)\in    \FF$,    with
    $c'_n  \leq  c_n$.  By  compactness,  along  a subsequence  we  have
    $(c'_n,\ell_n)    \to    (c'_\infty,\ell_\infty)\in   \FF$,    where
    $c'_\infty\leq                        c$,                        so:
    \(  \lossinf(c)  \leq   \lossinf(c'_\infty)  \leq  \ell_\infty.   \)
    Putting   the   inequalities   together    we   have   proved   that
    $\lossinf(c_n)$  converges   to  $\ell_\infty  =   \lossinf(c)$,  so
    $\lossinf$ is right-continuous.

      Eventually the equality $\lossinf(\maxcost)=0$ is obvious by
      definition of $\lossinf$. 
  \end{proof}

We start by a general result concerning the links between
the three optimization problems~\eqref{eq:multi-objective-minmax},
\eqref{eq:Prob1} and \eqref{eq:Prob2}. 

\begin{proposition}[Single-objective and bi-objective problems]\label{thm:single-bi}
  Suppose Assumption \ref{hyp:loss+cost} holds.
  \begin{enumerate}[(i)]
  \item\label{single-bi}  If   $\eta_\star$  is  Pareto   optimal,  then
    $\eta_\star$  is   a  solution  of~\eqref{eq:Prob1}  for   the  cost
    $c=C(\eta_\star)$, and  a solution of~\eqref{eq:Prob2} for  the loss
    $\ell=\loss(\eta_\star)$. Conversely, if  $\eta_\star$ is a solution
    to  both problems  \eqref{eq:Prob1}  and  \eqref{eq:Prob2} for  some
    values $c$ and $\ell$, then $\eta_\star$ is Pareto optimal.
    In particular, we have:
    \begin{equation}
      \label{eq:Pareto=}
      \cp=\left\{\eta\in \Delta\, \colon\, \loss(\eta)=\lossinf
        (C(\eta))\right\} 
      \,      \cap\, \left\{\eta\in \Delta\, \colon\, C(\eta)=\Cinf (\loss(\eta))\right\}.
\end{equation}

  \item\label{intersection} The Pareto frontier is the intersection of the graphs of
    $\Cinf$ and $\lossinf$: \[ \mathcal{F} = \{ (c,\ell)\in [0, \maxcost]\times [0,
    \maxloss]\, \colon\, c=\Cinf(\ell) \text{ and } \ell = \lossinf(c)\}. \]
  \item\label{special_points} The points $(0,\lossinf(0))$ and $(\Cinf(0),0)$ both belong
    to the Pareto frontier, and $\Cinf(\lossinf(0)) = \lossinf(\Cinf(0))=0$.
    Moreover, we also have $\Cinf(\ell) = 0$ for $\ell\in [\lossinf(0), \maxloss]$, and
    $\lossinf(c) = 0$ for $c\in [\Cinf(0), \maxcost]$.
  \end{enumerate}
\end{proposition}

\begin{proof}
  Let us prove~\ref{single-bi}. If $\eta_\star$ is Pareto optimal, then for any strategy
  $\eta$, if $C(\eta)\leq C(\eta_\star)$ then $\loss(\eta)\geq \loss(\eta_\star)$ by
  taking the contraposition in~\eqref{eq:defParetoOptimal}, and $\eta_\star$ is indeed a
  solution of Problem~\eqref{eq:Prob1} with $c=C(\eta_\star)$. Similarly $\eta_\star$ is a
  solution of Problem~\eqref{eq:Prob2}.

  For  the  converse  statement,  let  $\eta_\star$  be  a  solution  of
  \eqref{eq:Prob1}   for   some   $c$  and   of   \eqref{eq:Prob2}   for
  some~$\ell$.  It   is  also   a  solution  of   \eqref{eq:Prob1}  with
  $c=C(\eta_\star)$. In  particular, we  get that for  $\eta\in \Delta$,
  $\loss(\eta)<         \loss(\eta_\star)$          implies         that
  $C(\eta)>c=C(\eta_\star)$,    which   is    the    second   part    of
  \eqref{eq:defParetoOptimal}.  Similarly, use  that  $\eta_\star$ is  a
  solution  to  \eqref{eq:Prob2},   to  get  that  the   first  part  of
  \eqref{eq:defParetoOptimal}   also    holds.    Thus,    the   strategy
  $\eta_\star$ is Pareto optimal.

  \medskip

  To prove Point~\ref{intersection}, we first prove that $
  \mathcal{F} $ is a subset of $\{ (c,\ell)\, \colon\, c=\Cinf(\ell) \text{ and } \ell =
  \lossinf(c)\}$.
  A point in $\mathcal{F}$ may be written as $(C(\eta_\star),
  \loss(\eta_\star))$ for some Pareto optimal strategy $\eta_\star$.
  By Point~\ref{single-bi}, $\eta_\star$ solves
  Problem~\eqref{eq:Prob1}
  for the cost $C(\eta_\star)$, so $\lossinf(C(\eta_\star)) =
  \loss(\eta_\star)$. Similarly, we have
  $\Cinf(\loss(\eta_\star)) =
  C(\eta_\star)$, as claimed.

  We now prove the reverse inclusion. Assume that $c=\Cinf(\ell)$ and $\ell= \lossinf(c)$,
  and consider $\eta$ a solution of Problem~\eqref{eq:Prob2} for the loss~$\ell$:
  $\loss(\eta) \leq \ell$ and $C(\eta) = \Cinf(\ell) = c$. Then $\eta$ is admissible for
  Problem~\eqref{eq:Prob1} with cost $c=\Cinf(\ell)$, so $\loss(\eta) \geq
  \lossinf(\Cinf(\ell)) = \lossinf(c) = \ell$. Therefore, we get $\loss(\eta) =
  \lossinf(c)$, and $\eta$ is also a solution of Problem~\eqref{eq:Prob1}. By
  Point~\ref{single-bi}, $\eta$ is Pareto optimal, so $(C(\eta),\loss(\eta)) = (c,\ell)\in
  \mathcal{F}$, and the reverse inclusion is proved. \medskip

  Finally     we    prove     Point~\ref{special_points}.    We     have
  $\Cinf(0)=\min  \{ C(  \eta)\, \colon\,  \eta\in \Delta  \text{ and  }
  \loss(\eta)=0  \}\in [0,  \maxcost]$. Let  $\eta\in \Delta$  such that
  $\loss(\eta)=0$    and    $C(\eta)=\Cinf(0)$.     We    deduce    that
  $\lossinf(\Cinf(0))\leq \loss(\eta)=0$ and thus $\lossinf(\Cinf(0))=0$
  as  $\loss$ is  non-negative. We  deduce from~\ref{intersection}  that
  $(\Cinf(0),  0)$  belongs  to  $   \mathcal{F}  $.  Since  $\lossinf$  is
  non-increasing,   we   also   get  that   $\lossinf=0$   on   $[\Cinf(0),
  \maxcost]$. The  other properties  of \ref{special_points}  are proved
  similarly.
\end{proof}

\begin{figure}
  \begin{subfigure}[T]{.5\textwidth}
    \centering
    \includestandalone[page=1]{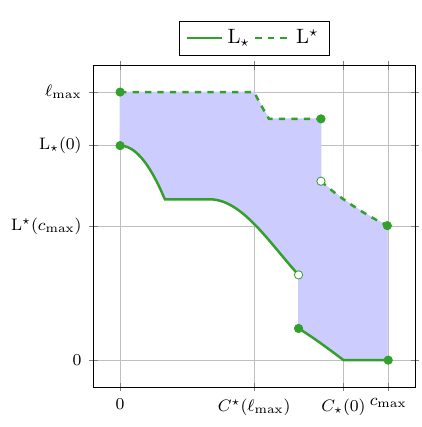}
    \caption{Value functions for Problems \eqref{eq:Prob1} and \eqref{eq:Prob1**}.}
    \label{fig:value_F}
  \end{subfigure}%
  \begin{subfigure}[T]{.5\textwidth} \centering
    \includestandalone[page=2]{frontiers_fig}
    \caption{Value functions for Problems \eqref{eq:Prob2} and \eqref{eq:Prob2**}.}\label{fig:frontier}
  \end{subfigure}
  \begin{subfigure}[T]{.5\textwidth}
    \centering
    \includestandalone[page=3]{frontiers_fig}
    \caption{Pareto and anti-Pareto frontier.}
  \end{subfigure}%
    \begin{subfigure}[T]{.5\textwidth}
      \centering
      \includestandalone[page=4]{frontiers_fig}
      \caption{Pareto and anti-Pareto frontier under additional regularity
      Assumptions~\ref{hyp:cost}-\ref{hyp:loss**}.}\label{fig:reg_frontier}
  \end{subfigure}
  \caption{An example of the possible aspects of the feasible region $\FF$ (in light blue), the value
 functions $\lossinf$, $\lossup$, $\Cinf$, $\Csup$, and the Pareto and anti-Pareto frontier (in red) under
 Assumption~\ref{hyp:loss+cost}.}
\label{fig:typical_frontier}
\end{figure}

The next two hypotheses on $C$ and $\loss$ will in particular help rule out ``flat parts''
in the Pareto frontier. 

\begin{hyp}\label{hyp:cost}
  Any local minimum of the cost function $C$ is global.
 \end{hyp}

\begin{hyp}\label{hyp:loss}
    Any local minimum of the loss function $\loss$ is global.
\end{hyp}

\begin{proposition}[Reduction to single-objective problems]
  \label{prop:f_properties}
  Under Assumptions \ref{hyp:loss+cost} and~\ref{hyp:cost}, the bi-objective problem
  may be reduced to cost optimization for a fixed loss; more precisely the following properties hold:
  \begin{enumerate}[(i)]
  \item\label{prop:c_star_decreases} The optimal cost $\Cinf$ is decreasing on $[0,
    \lossinf(0)]$.
  \item\label{prop:l_constraint_binding}
    For any $\ell\in [0,\lossinf(0)]$ there exist $\eta$ that
    solve Problem~\eqref{eq:Prob2}.
    For any such $\eta$,  $\loss(\eta) = \ell$ (that is, the constraint is binding). Moreover, $\eta$ is
    Pareto optimal, and:
    \begin{equation}\label{eq:LC=id}
      \lossinf(\Cinf(\ell)) = \ell.
    \end{equation}
    In particular, we have:
\begin{equation}
   \label{eq:Pareto=C} 
     \cp= \left\{\eta\in \Delta\, \colon\, C(\eta)=\Cinf (\loss(\eta))
       \quad\text{and}\quad
     \loss(\eta)\leq \lossinf(0)\right\}. 
    \end{equation}    
  \item\label{prop:frontier=graph_c_star} The Pareto frontier is the graph of $\Cinf$:
    \begin{equation}\label{eq:FL=L*}
      \mathcal{F}= \{(\Cinf(\ell), \ell) \, \colon \, \ell \in [0,\lossinf(0)]\}.
    \end{equation}
  \end{enumerate}
  Similarly, under Assumptions \ref{hyp:loss+cost} and~\ref{hyp:loss},
  the bi-objective problem may be reduced to loss optimization for a fixed cost, that is:
  \begin{enumerate}[(i),resume]
  \item\label{prop:l_star_decreases} The optimal loss $\lossinf$ is decreasing on $[0,
    \Cinf(0)]$.
  \item\label{prop:c_constraint_binding}
    For any $c\in[0,\Cinf(0)]$ there exist $\eta$ that solve Problem~\eqref{eq:Prob1}.
    For any such $\eta$,  $C(\eta) =c$. Moreover $\eta$ is Pareto
    optimal, and $\Cinf(\lossinf(c)) = c$.
      In particular, we have:
    \begin{equation}
      \label{eq:Pareto=L}
      \cp=\left\{\eta\in \Delta\, \colon\, \loss(\eta)=\lossinf
        (C(\eta)) \quad\text{and}\quad
        C(\eta) \leq  \Cinf(0)\right\}.
    \end{equation}

  \item\label{prop:frontier=graph_l_star} The Pareto frontier is the graph of $\lossinf$:
    \begin{equation}\label{eq:FL=C*} \mathcal{F} = \{(c, \lossinf(c)) \, \colon \, c
      \in [0,\Cinf(0)]\}.
    \end{equation}
  \end{enumerate}
\end{proposition}

\begin{proof}
  We prove \ref{prop:c_star_decreases}. Let $0\leq \ell<\ell' \leq \lossinf(0)$, and let
  $\eta_\star$ be a solution of Problem~\eqref{eq:Prob2}:
  \begin{equation}
    \label{eq:etaSolvesProb2}
    C(\eta_\star) = \Cinf(\ell) \quad\text{and}\quad \loss(\eta_\star)
    \leq \ell.
  \end{equation}
  The set $\mathcal{O} = \{\eta\, \colon\,  \loss(\eta) < \ell'\}$ is open and contains $\eta_\star$.
  Since $\loss(\eta_\star)<\lossinf(0)$, we get $C(\eta_\star)>0$, so $\eta_\star$ is not
  a global minimum for $C$. By Assumption~\ref{hyp:cost}, it cannot be a local minimum for
  $C$, so~$\mathcal{O}$ contains at least one point $\eta'$ for which
  $C(\eta')<C(\eta_\star)$. Since $\eta'\in \mathcal{O}$, we get $\loss(\eta')\leq \ell'$,
  so that $\Cinf(\ell') \leq C(\eta') < C(\eta_\star)= \Cinf(\ell)$. Since $\ell<\ell'$
  are arbitrary, $\Cinf$ is decreasing on $[0,\lossinf(0)]$. 

  We now prove~\ref{prop:l_constraint_binding}. If the inequality
  in~\eqref{eq:etaSolvesProb2} was strict, that is $\loss(\eta_\star)<\ell$, then we would
  get a contradiction as $C(\eta_\star) \geq \Cinf(\loss(\eta_\star)) > \Cinf(\ell) =
  C(\eta_\star)$. Therefore any solution $\eta_\star$ of \eqref{eq:Prob2} satisfies
  $\loss(\eta_\star) = \ell$, and in particular $\Cinf(\loss(\eta_\star))=\Cinf(\ell) =
  C(\eta_\star)$. This implies in turn that $\eta_\star$ also solves~\eqref{eq:Prob1}: if
  $\eta$ satisfies $\loss(\eta)<\loss(\eta_\star)$, then using the definition of $\Cinf$,
  the fact that it decreases, and the definition of $\eta_\star$, we get:
  \[
    C(\eta) \geq \Cinf(\loss(\eta)) > \Cinf(\loss(\eta_\star)) = C(\eta_\star).
  \]
  By contraposition, we have $\loss(\eta)\geq \loss(\eta_\star)$ for any $\eta$ such that
  $C(\eta)\leq C(\eta_\star)$, proving that $\eta_\star$ is also a solution
  of~\eqref{eq:Prob1} with $c=C(\eta_*)$. By Point~\ref{single-bi} of
  Proposition~\ref{thm:single-bi}, $\eta_\star$ is Pareto optimal. Therefore
  $(C(\eta_\star),\loss(\eta_\star)) = (\Cinf(\ell),\ell)$ belongs to the Pareto frontier.
  Using Point~\ref{intersection} of Proposition~\ref{thm:single-bi}, we deduce that
  $\ell=\lossinf(\Cinf(\ell))$.

  \medskip

  To  prove Point~\ref{prop:frontier=graph_c_star},  note that  Equation
  \eqref{eq:LC=id}     shows     that,    if     $c=\Cinf(\ell)$     for
  $\ell\in    [0,   \lossinf(0)]$,    then   $\ell=\lossinf(c)$.     Use
  Point~\ref{intersection}       and       \ref{special_points}       of
  Proposition~\ref{thm:single-bi},          to         get          that
  $\mathcal{F} =  \{ (c,\ell)\,\colon\,  c=\Cinf(\ell), \, \ell  \in [0,
  \lossinf(0)]\}$. \medskip

  The claims \ref{prop:l_star_decreases}, \ref{prop:c_constraint_binding} and
  \ref{prop:frontier=graph_l_star} are proved in the same way, exchanging the roles of
  $\loss$ and~$C$.
\end{proof}

Under both hypotheses, the picture becomes much nicer, see Fig.~\ref{fig:reg_frontier},
where the only flat parts of the graphs of $\Cinf$ and $\lossinf$
occur at zero loss or zero cost.

\begin{proposition}[When all problems are equivalent]
  \label{prop:f_properties-2}
  Under          Assumptions~\ref{hyp:loss+cost},         \ref{hyp:cost}
  and~\ref{hyp:loss},  the following properties hold:
  \begin{enumerate}[(i),resume]
  \item \label{item:inv-inf}
    The optimal loss $\lossinf$  is  a  continuous  decreasing
  bijection of $[0,\Cinf(0)]$ onto $[0,  \lossinf(0)]$ and $\Cinf$ is the
  inverse bijection. 

\item \label{it:prop-Fr} The Pareto frontier $\mathcal{F}$ is compact and
  connected.
  \item \label{it:P=K} The set of Pareto optimal strategies $\mathcal{P}$ is compact.
  \end{enumerate}
\end{proposition}

\begin{proof}  
  We  first  check that  $\Cinf$  and  $\lossinf$ are  continuous  under
  Assumptions~\ref{hyp:loss+cost},~\ref{hyp:cost} and~\ref{hyp:loss}. We
  deduce                                                            from
  Proposition~\ref{prop:f_properties}~\ref{prop:l_constraint_binding}
  and Proposition~\ref{prop:main-result}  that $[0, \lossinf(0)]$  is in
  the range  of $\lossinf$. Since  $\lossinf $ is decreasing,  thanks to
  Proposition~\ref{prop:f_properties}~\ref{prop:l_star_decreases}    and
  $\lossinf(\Cinf(0))=0$,                                            see
  Proposition~\ref{thm:single-bi}~\ref{special_points},              the
  function~$\lossinf$     is     continuous    and     decreasing     on
  $[0,  \lossinf(0)]$,  and  thus  
one-to-one         from     $[0,\Cinf(0)]$      onto
  $[0,  \lossinf(0)]$.    Thanks  to~\eqref{eq:LC=id},   its  inverse
  bijection is the function~$\Cinf$.

  Since  the frontier  $\mathcal{F}$ is  given by~\eqref{eq:FL=C*},  and
  $\lossinf$ is  continuous, the  frontier $\mathcal{F}$ is  compact and
  connected.

  Since     $\mathcal{F}$       is      closed and 
  $\mathcal{P}=f^{-1}(\mathcal{F})$,  where  the function  $f=(C,\loss)$
  defined on $\Delta$  is continuous, we deduce that  $\mathcal{P}$ is a
  closed subset of $\Delta$ and thus compact as $\Delta$ is compact.
\end{proof}

Finally, let us check that Assumptions~\ref{hyp:cost} and~\ref{hyp:loss}
hold under very simple assumptions, which are in particular satisfied
for the SIS model with vaccination by
the cost  functions $\costu$ and  $\costa$ and the loss  functions $R_e$
and $\I$, see Section~\ref{sec:prop-C+L}. 

\begin{lemma}\label{lem:c-dec+L-hom}
  Suppose Assumption \ref{hyp:loss+cost} holds. If the cost function $C$
is convex or decreasing, then Assumption
  \ref{hyp:cost}  holds,   and  in   the  latter   case  we   also  have
  $\lossinf   (0)=\maxloss$.   If   the   loss   function   $\loss$   is
  sub-homogeneous, then Assumption \ref{hyp:loss} holds.
\end{lemma}

\begin{proof}
  Let $\eta\in\Delta$ be a local minimum of $C$.  For $\varepsilon>0$
  small enough, we get $C(\eta+\varepsilon(\un-\eta)) \geq C(\eta)$
  using that
  $\eta+\varepsilon(\un-\eta) =(1-\varepsilon) \eta + \varepsilon \un$
  belongs to $ \Delta$ as $\Delta$ is convex.
  If $C$ is convex, using
  that $C(\un)=0$, we also get
  $C(\eta)\geq (1-\varepsilon) C(\eta)\geq C(\eta+\varepsilon(\un-\eta))$.
  Therefore all inequalities are equalities, so $C(\eta)=0$.
  We thus deduce that $\eta$ is
  a global minimum for $C$.  If $C$ is decreasing, using that
  $ \eta\preceq\eta+\varepsilon(\un-\eta) $, we get
  $C(\eta)\geq C(\eta+\varepsilon(\un-\eta))$, with the equality only
  if $\eta = \un$; we deduce that $\eta$ is a global minimum of $C$.
  So if the cost function $C$ is convex or decreasing, then Assumption~\ref{hyp:cost} holds. If $C$ is decreasing, as any local minimum is
  in fact equal to $\un$, this also gives $\lossinf(0)=\maxloss$.

  Similarly  if  $\loss$  has  a  local  minimum  at  $\eta$,  then  for
  $\varepsilon>0$         small          enough,         we         have
  $\loss(\eta)  \leq  \loss((1-\varepsilon) \eta)  \leq  (1-\varepsilon)
  \loss(\eta)$,  so $\loss(\eta)  = 0$,  where we  used that  $\loss$ is
  sub-homogeneous  for the  second  equality. Thus  $\eta$  is a  global
  minimum of $\loss$.
\end{proof}

\subsection{On the anti-Pareto frontier}\label{sec:Pareto-AF}

Letting $C'(\eta) = \maxloss - \loss(\eta)$ and $\loss'(\eta) = \maxcost -
C(\eta)$, it is easy to see that:
\[
 \Cinf'(c) =\maxloss - \lossup(\maxcost-c)
 \quad\text{and}\quad
 \lossinf'(\ell) = \maxcost - \Csup(\maxloss - \ell),
\]
so that Propositions~\ref{prop:main-result} and~\ref{thm:single-bi} may be applied to the cost function
$C'$ and the loss function $\loss'$ to yield
the following result.

\begin{proposition}[Single-objective and bi-objective problems ---  the anti
 Pareto case]
 \label{thm:single-bi-anti}
 Suppose Assumption \ref{hyp:loss+cost} holds.
\begin{enumerate}[(1)]
  \setcounter{enumi}{-1}
\item \label{it:exist-c}
      For any cost $c\in[0,\maxcost]$, there exists a
    maximizer of  the loss under  the cost constraint  $C(\cdot)\geq c$,
    that is,  a solution  to Problem~\eqref{eq:Prob1**}.  Similarly, for
    any loss $\ell\in[0,\maxloss]$, there exists a maximizer of the cost
    under  the  loss constraint  $\loss(\cdot)  \geq  \ell$, that  is  a
    solution to Problem~\eqref{eq:Prob2**}.
 \end{enumerate}
\begin{enumerate}[(i)]
 \item\label{single-bi-anti}
 If $\eta ^\star$ is anti-Pareto optimal, then $\eta^\star$ is a solution of~\eqref{eq:Prob1**}
 for the cost $c=C(\eta^\star)$, and a solution of~\eqref{eq:Prob2**} for the loss
 $\ell=\loss(\eta^\star)$.
 Conversely, if $\eta^\star$ is a solution to both problems
 \eqref{eq:Prob1**}
 and \eqref{eq:Prob2**} for some values $c$ and $\ell$,
 then $\eta^\star$ is anti-Pareto optimal.
   In particular, we have:
    \begin{equation}
      \label{eq:Pareto-anti=}
      \cpa=\left\{\eta\in \Delta\, \colon\, \loss(\eta)=\lossup (C(\eta))\right\}
\,      \cap\, \left\{\eta\in \Delta\, \colon\, C(\eta)=\Csup (\loss(\eta))\right\}.
    \end{equation}

\item\label{intersection-anti}
 The anti-Pareto frontier is the intersection of the graphs of $\Csup$
 and $\lossup$:
 \begin{equation}
   \label{eq:FL=L*-anti}
  \AF = \{ (c,\ell)\in [0, \maxcost]\times [0,
 \maxloss]\, \colon\, c=\Csup(\ell) \text{ and } \ell =
 \lossup(c)\}.
 \end{equation}

 \item\label{special_points-anti}
 The points $(\Csup(\maxloss),\maxloss)$ and $(\maxcost,\lossup(\maxcost))$ both belong
 to the anti-Pareto frontier, and we have
 $\Csup(\lossup(\maxcost)) =\maxcost$ and $\lossup(\Csup(\maxloss)) =
 \maxloss$.
 Moreover, we also have $\Csup(\ell) = \maxcost$ for $\ell\in
 [0, \lossup(\maxcost)]$,
 and $\lossup(c) = \maxloss$ for $c\in [0,\Csup(\maxloss)]$.
 \end{enumerate}
\end{proposition}

The following additional hypotheses rule out the occurrence of flat
parts in the anti-Pareto frontier.
\begin{hyp}
 \label{hyp:cost**}
 Any local maximum of the cost function $C$ is global.
 \end{hyp}

\begin{hyp}\label{hyp:loss**}
    Any local maximum of the loss function $\loss$ is global.
\end{hyp}

The following result is now a consequence of
Proposition~\ref{prop:f_properties} and Proposition~\ref{prop:f_properties-2}
applied to the loss function $\loss'$ and cost function $C'$.

\begin{proposition}[Reduction to single-objective problems --- the anti-Pareto case]
  \label{prop:f_properties**}
  Under Assumption \ref{hyp:loss+cost} and~\ref{hyp:cost**} the following properties hold:
  \begin{enumerate}[(i)]
  \item\label{prop:c_star_decreases**} The optimal cost $\Csup$ is decreasing on
        $[\lossup(\maxcost), \maxloss]$.
  \item\label{prop:l_constraint_binding**} If $\eta$ solves Problem~\eqref{eq:Prob2**} for
    the loss $\ell\in[\lossup(\maxcost),\maxloss]$, then $\loss(\eta) = \ell$ (that is,
    the constraint is binding). Moreover $\eta$ is anti-Pareto optimal, and
    $\lossup(\Csup(\ell)) = \ell$. We have:
       \begin{equation}
      \label{eq:anti-P=C*}
      \cpa=\{ \eta\in \Delta\, \colon\, C(\eta)=\Csup(\loss(\eta)) 
      \quad\text{and}\quad
      \loss(\eta) \in [\lossup(\maxcost),\maxloss]\}.
    \end{equation}
  \item\label{prop:frontier=graph_c_star**} The anti-Pareto frontier is the graph of
    $\Csup$:
    \begin{equation}\label{eq:FL=L**} \AF = \{(\Csup(\ell),
	\ell) \, \colon \, \ell \in
      [\lossup(\maxcost),\maxloss]\}.
    \end{equation}
  \end{enumerate}
  Similarly, under Assumptions \ref{hyp:loss+cost} and~\ref{hyp:loss**}, the following
  properties hold:
  \begin{enumerate}[(i),resume]
  \item\label{prop:l_star_decreases**} The optimal loss $\lossup$ is decreasing on
     $[\Csup(\maxloss), \maxcost]$. 
  \item\label{prop:c_constraint_binding**} If $\eta$ solves Problem~\eqref{eq:Prob1**} for
    the cost $c\in[\Csup(\maxloss),\maxcost]$, then $C(\eta) =c$. Moreover $\eta$ is
    anti-Pareto optimal, and $\Csup(\lossup(c)) = c$. We have:
       \begin{equation}
      \label{eq:anti-P=L*-Re}
      \cpa=\{ \eta\in \Delta\, \colon\, \loss(\eta)=\lossup(C(\eta)) 
      \quad\text{and}\quad
      C(\eta) \in [\Csup(\maxloss),\maxcost]\}.
    \end{equation}

  \item\label{prop:frontier=graph_l_star**} The anti-Pareto frontier is the graph of
    $\lossup$:
    \begin{equation}\label{eq:FL=C**} \AF = \{(c,
      \lossup(c)) \, \colon \, c \in [\Csup(\maxloss),\maxcost]\}.
    \end{equation}
  \end{enumerate}
  Finally,  if the three Assumptions  \ref{hyp:loss+cost},  \ref{hyp:cost**}  and
  \ref{hyp:loss**}  hold,  then  $\lossup$ is  a  continuous  decreasing
  bijection from~$[\Csup(\maxloss),\maxcost]$         onto
  $ [\lossup(\maxcost),\maxloss]$, $\Csup$ is the inverse bijection, the
  anti-Pareto  frontier  $\AF$  is  compact  and
  connected,   and   the   set   of   anti-Pareto   optimal   strategies
  $\cpa$ is compact.
\end{proposition}

The following result is similar to the first part of Lemma \ref{lem:c-dec+L-hom}.

\begin{lemma}\label{lem:c-dec+L-hom**}
  Suppose the part of Assumption~\ref{hyp:loss+cost} on the cost
  function holds.  If the cost function $C$
  is concave
  or decreasing, then Assumption~\ref{hyp:cost**}  holds,   and  in   the  latter   case  we   also  have
  $\lossup (\maxcost)=0$.
\end{lemma}

\begin{proof}
  Let  $\eta\in\Delta$  be a local maximum of $C$  and  $\varepsilon\in
  (0,1)$ small enough so that $C(\eta)\geq  C((1-\varepsilon)\eta)$.
If $C$ is concave, we also get $C((1-\varepsilon)\eta) \geq (1-\varepsilon)
C(\eta)+ \varepsilon C(\zero)= (1-\varepsilon)
C(\eta)+ \varepsilon \maxcost$. We deduce that $C(\eta)=\maxcost$ and thus
$\eta$ is a global maximum.

If    $C$  is
  decreasing,  we get $C((1-\varepsilon)\eta) \geq  C(\eta)$, with  equality if
  and only if $\eta=\zero$. Therefore the only local
  maximum of  $C$ is  $\eta=\zero$, and  it is  a global  maximum. Since
  $C(\eta)=\maxcost$ implies that  $\eta=\zero$,
  we also get that $\lossup(\maxcost)=\loss(\zero)=0$.
\end{proof}

In  the  SIS  model,  the  loss function  $\loss=\I$  does  not  satisfy
Assumption~\ref{hyp:loss**}  in  general  because   of  its  zeros,  see
Remark~\ref{rem:local-max-I},  but  it  satisfies the  following  weaker
condition.

  \begin{chyp}{\ref{hyp:loss**}'}
    \label{hyp:loss=I}
     If $\eta$ is a local maximum of  $\loss$ such that
     $\loss(\eta)>0$, then it is a global  maximum.

 \end{chyp}

Under this assumption, the function $\Csup$ might be discontinuous at
$0$. We set:
\begin{equation}
   \label{eq:def-c+}
  \cplus=\Csup(0+) = \lim_{\substack{\ell \to 0\\ \ell >0}} \Csup(\ell).
\end{equation}
Considering Assumption~\ref{hyp:loss=I} with $\cplus<\maxcost$ instead
of Assumption~\ref{hyp:loss**} impacts only
items~\ref{prop:l_star_decreases**}-\ref{prop:frontier=graph_l_star}
and the conclusion of Proposition~\ref{prop:f_properties**} as
follows; we refer to Fig.~\ref{fig:monatomic} for an illustration,
and leave the proof to the reader.

\begin{cprop}{\ref{prop:f_properties**}'}[The anti-Pareto case, weaker version]
  \label{prop:f_properties=I}
  Suppose Assumptions \ref{hyp:loss+cost}  and~\ref{hyp:loss=I} hold.
\begin{enumerate}[(a)]
\item\label{it:a}
  If
  $\cplus=\maxcost$,                                                then
  items~\ref{prop:l_star_decreases**}-\ref{prop:frontier=graph_l_star}
  and the conclusion of Proposition~\ref{prop:f_properties**} hold.
\item \label{it:b}
  If $\cplus<\maxcost$, the following properties hold:
    \begin{enumerate}[(i),start=4]
   \item\label{prop:l_star_decreases=I} The optimal loss $\lossup$ is decreasing on
     $[\Csup(\maxloss), \cplus]$ and zero on $(\cplus, \maxcost]$.
  \item\label{prop:c_constraint_binding=I} If $\eta$ solves Problem~\eqref{eq:Prob1**} for
    the cost $c\in[\Csup(\maxloss),\cplus)\cup\{\maxcost\}$, then $C(\eta) =c$. Moreover $\eta$ is
    anti-Pareto optimal, and $\Csup(\lossup(c)) = c$. We have:
       \begin{equation}
      \label{eq:anti-P=L*-I}
      \cpa=\{ \eta\in \Delta\, \colon\, \loss(\eta)=\lossup(C(\eta)) 
      \quad\text{and}\quad
      C(\eta) \in [\Csup(\maxloss),\cplus) \cup\{\maxcost\}\}.
    \end{equation}

  \item\label{prop:frontier=graph_l_star=I} The anti-Pareto frontier is
    also given by:
    \begin{equation}\label{eq:FL=C=I} \AF = \{(c,
      \lossup(c)) \, \colon \, c \in
      [\Csup(\maxloss),\cplus)\}\cup\{(\maxcost,
      0)\}.
    \end{equation}
  \end{enumerate}
  If furthermore Assumption~\ref{hyp:cost**} holds, then
    $\lossup(\cplus)=0$, 
  the function $\lossup$ is a
  continuous  decreasing  bijection of  $[\Csup(\maxloss),\cplus)$  onto
  $ (0,\maxloss]$, $\Csup$ is the inverse bijection, and
  the  union of the anti-Pareto  frontier   with its limit point~$\{(\cplus, 0)\}$  is
  compact but not connected.
\end{enumerate}
\end{cprop}

  Clearly, the inequality $\cplus\leq \maxcost$ always holds.
  The rest of this section discusses conditions for, and consequences of,
  a strict inequality. In particular it 
  helps decide which of the cases \ref{it:a} and \ref{it:b} holds
  in the previous proposition.

\begin{lemma}[On $\cplus< \maxcost$]
  \label{lem:c0<cmax}
 Suppose Assumption~\ref{hyp:loss+cost} holds. If the cost function is decreasing
 and if $\{\eta\in
       \Delta\, \colon\, \loss(\eta)=0\}$ contains a neighborhood of
       $\zero$ in $\Delta$, then we have $\cplus<\maxcost$. 
\end{lemma}
\begin{proof}
  Let $O$ be  an open neighborhood of $\zero$  in $\Delta$ such that
  $\loss=0$  on  $O$.   Since  the  cost  function   is  continuous  and
  $\Delta\setminus  O$  is  compact,  it  reaches  it  maximum,  say  at
  $\eta$. Since $\eta\neq  \zero$ and the cost is  decreasing, we deduce
  that    $C(\eta)<   C(\zero)=\maxcost$.    This   gives    also   that
  $\cplus\leq   \sup\{C(\eta')\,    \colon\,   \eta'\in   \Delta\setminus
  O\}=C(\eta)<\maxcost$.
\end{proof}

For  $\eta_1, \eta_2\in E$, we denote by $\eta_1 \vee \eta_2$
the maximum, when it exists,  of $\eta_1$
and   $\eta_2$: it is then the only element $\eta\in  E$ such that 
$\eta_i\preceq \eta$ for $i=1,2$ and, 
  for  any   $\eta'\in  E$  satisfying
$   \eta_i\preceq  \eta'$  for $i=1,2$,  we  also   have  that
$ \eta\preceq \eta'$. 

\begin{lemma}[On Assumptions~\ref{hyp:loss**} and~\ref{hyp:loss=I}]
  \label{lem:5-et-5'}
  Suppose Assumption~\ref{hyp:loss+cost} holds.
  \begin{enumerate}[(i)]
   \item\label{it:5=}  If Assumption~\ref{hyp:loss**} holds also, then we
     have $\cplus=\maxcost$.
   \item\label{it:5'=} Assume  that the  map $(\eta, \eta')  \mapsto \eta  \vee \eta'$
     from  $\Delta\times \Delta$  to  $\Delta$ is  well  defined and,
     for  any  fixed~$\eta'$,   is
     continuous   at  $\eta=\zero$.   Assume
     furthermore that the   cost  function is   decreasing. If 
     Assumption~\ref{hyp:loss=I}   holds  and   $\cplus=\maxcost$,  then
     Assumption~\ref{hyp:loss**} holds.
  \end{enumerate}
\end{lemma}

\begin{proof}
   Suppose Assumption~\ref{hyp:loss+cost} holds.    First, let us check that  $\cplus=\maxcost$ under  Assumption~\ref{hyp:loss**}. On the one hand, if $\Csup(\maxloss)=\maxcost$
  (pathological case), then we have $\Csup(\ell)\geq
\Csup(\maxloss)=\maxcost$, and, since by
definition  $\Csup$ is
non-increasing and $\maxcost=\Csup(0)$,  we get  $\cplus=\maxcost$.
  On the other hand,  if  $\Csup(\maxloss)\leq  \cplus<\maxcost$,  thanks to Assumptions~\ref{hyp:loss+cost}
and~\ref{hyp:loss**},  for $c\in (\cplus,
\maxcost)$, there exists a strategy which solves Problem~\eqref{eq:Prob1**} for
the cost $c$ (by Proposition~\ref{thm:single-bi-anti}~\ref{it:exist-c}),
and it is then anti-Pareto (by
Proposition~\ref{prop:f_properties**}~\ref{prop:c_constraint_binding**}),
so  we also have 
 $\Csup(\lossup(c))=c>
\cplus=\Csup(0+)$.
We deduce that $\lossup(c)=0$, which is absurd since $\lossup$
is decreasing on $[\Csup(\maxloss), \maxcost]$. Thus, we have   $\cplus=
\maxcost$, which  proves
Point~\ref{it:5=}.

\medskip
We now prove Point~\ref{it:5'=}.
If $\eta'$ is a local
maximum of $\loss$ such that $\loss(\eta')=0$, since $\loss(\eta) \leq
\loss(\eta\vee \eta')$ as the loss is non-decreasing and since the map 
  $(\eta, \eta')  \mapsto \eta  \vee \eta'$
defined on $\Delta\times \Delta$ is continuous at $\eta=\zero$, we
deduce there exists an open  neighborhood $O$  of $\zero$ in $\Delta$
such that $\loss=0$ on $O$.
By Lemma~\ref{lem:c0<cmax}, we deduce that $\cplus<\maxcost$. 
In
conclusion, if  $\cplus=\maxcost$, we deduce there exists no local
maximum of $\loss$ being a zero of $\loss$, and thus Assumption~\ref{hyp:loss**} holds.
\end{proof}

\begin{remark}[On the continuity assumption in Lemma~\ref{lem:5-et-5'}~\ref{it:5'=}]
  \label{rem:5-et-5'}
  The map $(\eta,\eta') \to \eta\vee \eta'$ is not necessarily continuous
  in general.

  Following \cite[Definition~II.1.2]{schaefer_banach_1974},  recall that
  a  vector space  $E$ endowed  with an  order relation  $\preceq$ is  a
  \emph{vector lattice} if the order is compatible with the vector space
  structure and if,  for $x, y\in E$,  $x \vee y$ exists  (as the unique
  element             of            $E$             such            that
  $(x\preceq  z,  y\preceq  z)\implies   (x\vee  y)\preceq  z$  for  all
  $z$).   Note   that   $x   \wedge   y$  may   then   be   defined   as
  $-  (-x) \vee  (-y)$. If  $E$ is  also equipped  with a  norm that  is
  compatible with  the order in  the sense that  $|\eta|\preceq |\eta'|$
  implies  $\norm{\eta}\leq \norm{\eta'}$,  then $E$  is a  \emph{normed
    vector lattice}, and the map $(\eta, \eta') \mapsto \eta \vee \eta'$
  is               then              uniformly               continuous,
  see~\cite[Proposition~II.5.2]{schaefer_banach_1974}.

  However, this map $(\eta, \eta') \mapsto  \eta \vee \eta'$ may fail to
  be  continuous even  if  $(E,  \norm{\cdot})$ is  a  Banach space  and
  $(E, \preceq)$ a vector lattice. Indeed, consider the vector space $E$
  of real valued  Lipschitz functions defined on  $[0,1]$, equipped with
  the  norm $\norm{f}=\sup  |f|+ L(f)$,  where $L(f)$  is the  Lipschitz
  constant
  $L(f)= \sup\{ |f(x)-f(y)|/(x-y)\, \colon \, 0\leq y <x\leq 1\}$.  This
  makes $(E, \norm{\cdot})$ a Banach space,  which we order by the usual
  order $\leq $  on functions: $f\leq g$ if $\inf  (g-f)\geq 0$, so that
  $(E,  \leq )$  is a  vector lattice.  Denoting by  $\Id$ the  identity
  function,  let    $\eta'=\un   -   \Id$ and $\eta_\varepsilon= \varepsilon  \Id$  for
  $\varepsilon\in [0, 1]$.   The   sequence
  $(\eta_\varepsilon,  \varepsilon>0)$  clearly  converges  in  norm  to
  $\zero$    as     $\varepsilon$    goes    down    to     zero,    but
   $\norm{  \eta_\varepsilon  \vee  \eta' -\eta'}  =1+2\varepsilon$,  so  the
  function $\eta \mapsto \eta \vee \eta'$ is not continuous at $\zero$.

  Notice in this example that the set
  $\Delta=\{\eta\in E\, \colon\, \zero \leq \eta \leq \un\}$ is convex
  with $\zero=\min \Delta$ and $ \un=\max \Delta$, but  not
  compact. Therefore $\Delta$  does not satisfy Condition~\eqref{eq:D-prop}.  By
  construction the set $\Delta$ is   saturated, so that  the map
  $(\eta, \eta') \mapsto \eta \vee \eta'$ is well defined from
  $\Delta \times \Delta$ to $\Delta$; however, the map
  $\eta \mapsto \eta \vee \eta'$ restricted to $\Delta$ is not
  continuous at $\zero$ for $\eta'=\un -\Id\in \Delta$.
\end{remark}

\begin{remark}[Continuity holds for the SIS model]
  \label{rem:5-et-5'-SIS}
  In    the    application   of    the    SIS    model   presented    in
  Section~\ref{sec:k-SIS-model},  there is  no  norm  associated to  the
  topological vector space $E=L^\infty(\Omega)$  endowed with the weak-*
  topology, so  we cannot  appeal to the  general continuity  result for
  normed vector lattices.  Notice that $E$ endowed with  the usual order
  $\leq   $   on    functions   is   a   vector    lattice.    The   set
  $\Delta=\{\eta\in  E\,  \colon\,  \zero   \leq  \eta  \leq  \un\}$  is
  saturated and thus the map $(\eta,  \eta') \mapsto \eta \vee \eta'$ is
  well defined on $\Delta\times \Delta$,  and for any $\eta'\in \Delta$,
  the  map  $\eta  \mapsto  \eta  \vee \eta'$  defined  on  $\Delta$  is
  continuous  at $\eta=\zero$  thanks to  Lemma~\ref{lem:cont}.  On  one
  hand,  an elementary  extension of  this lemma  entails that,  for any
  $\eta'\in \Delta$, the map $\eta\mapsto \eta \vee \eta'$ is continuous
  on   $\Delta$  at   $\eta$   if  $\eta=\un_A$   for  some   measurable
  $A\subset \Omega$. On the other hand, it  is not hard to check that if
  $\Omega=[0,  1]$   and  $\mu$   is  the   Lebesgue  measure,   and  if
  $\eta_0\in           \Delta$           is          such           that
  $\int_\Omega   \eta_0(1-\eta_0)\,  \rd   \mu>0$,  then   there  exists
  $\eta'\in \Delta$,  such that  the map  $\eta\mapsto \eta  \vee \eta'$
  defined  on $\Delta$  is not  continuous at  $\eta_0$ (Hint:  consider
  $\eta'=\eta=  c  \un   $  with  $c\in  (0,  1/2]$   and  the  sequence
  $\eta_n= \eta +  (c/2) \, \sin(2 n \pi \cdot)$  which are all elements
  of  $\Delta$; by  the  Riemann-Lebesgue lemma,  we  get that  $\eta_n$
  weakly-*  converges  to  $\eta=\eta'$   whereas  $\eta_n  \vee  \eta'$
  weakly-*   converges  to   $(1+(2\pi)^{-1})  \eta'$;   thus  the   map
  $\eta\mapsto \eta \vee \eta'$ defined on $\Delta$ is not continuous).

To sum up, the continuity condition in
Lemma~\ref{lem:5-et-5'}~\ref{it:5'=} is satisfied in the SIS model
even though the map  $\eta \mapsto \eta \vee \eta'$ is not
continuous in general on $\Delta$ for all $\eta'\in \Delta$. 
\end{remark}

\section{Miscellaneous properties of the set of outcomes and the Pareto
frontier}\label{sec:diversCL}

We prove results concerning the feasible region, the stability of the Pareto frontier and
its geometry.

\subsection{The feasible region}

In the following proposition we check a number of topological properties of the set of
outcomes~$\FF = \{(C(\eta), \loss(\eta)), \eta\in\Delta\}$. 

\begin{proposition}[No hole in the feasible region]\label{prop:trou}
  Suppose that Assumption \ref{hyp:loss+cost} holds. The feasible region~$\FF$ is compact,
  path connected, and its complement is connected in $\R^2$. It is the whole region
  between the graphs of the one-dimensional value functions:
 \begin{equation}
   \label{eq:FF=}
 \begin{aligned}
    \FF &= \{ (c,\ell) \in \R^2 \,\colon\, 0\leq c \leq \maxcost,\,
    \lossinf(c) \leq \ell \leq \lossup(c) \} \\
	&= \{ (c,\ell) \in \R^2 \,\colon\, 0\leq \ell \leq \maxloss,\,
	\Cinf(\ell) \leq c \leq \Csup(\ell)\}.
  \end{aligned}
 \end{equation}
\end{proposition}

\begin{proof}
 Recall that $\Delta$ is convex and thus 
  path-connected.     The region $\FF$ is compact
  and  path-connected  as  a  continuous image  by  $(C,\loss)$  of  the
  compact, path-connected set $\Delta$.

  By symmetry, it is enough to prove that $\FF$ is equal to
  $F_1=\{ (c,\ell) \in \R^2 \,\colon\, 0\leq c \leq \maxcost,\, \lossinf(c) \leq \ell \leq
  \lossup(c) \} $. Let $(c,\ell) \in \FF$ and $\eta\in \Delta$ be such that $(c,\ell) =
  (C(\eta),\loss(\eta))$. By definition of $\lossinf$ and $\lossup$, we have: $
  \lossinf(c) = \lossinf(C(\eta)) \leq \loss( \eta) \leq \lossup(C(\eta)) = \lossup(c)$.
  We deduce that $(c,\ell)\in F_1$.

  \medskip

  Let us now prove that $F_1\subset  \FF$. Let us first consider a point
  of  the form  $(c, \lossinf(c))$,  where $0\leq  c \leq  \maxcost$. By
  definition,  there  exists  $\eta$  such that  $C(\eta)  \leq  c$  and
  $\loss(\eta)   =  \lossinf(c)$.   Let  $\eta_t   =  t\eta$.   The  map
  $t\mapsto    C(\eta_t)$     is    continuous    from     $[0,1]$    to
  $[C(\eta),\maxcost]$. As  $c\in[C(\eta),\maxcost]$,  there  exists
  $s$  such  that $C(\eta_s)  =  c$.  Since $\loss$  is  non-decreasing,
  $\loss(\eta_s)\leq  \loss(\eta)$.  By   definition  of  $\lossinf(c)$,
  $\loss(\eta_s)\geq               \lossinf(c)$.               Therefore
  $(c,\lossinf(c))     =    (C(\eta_s),\loss(\eta_s))$     belongs    to
  $\FF$. Similarly the graphs of $\Cinf$, $\Csup$ and $\lossup$ are also
  included in $\FF$.

  So, it is enough to check that, if $A=(c, \ell)$ is in $F_1$, with $c\in(0,\maxcost)$
  and $\ell\in(\lossinf(c),\lossup(c))$, then $A$~belongs to $\FF$. We shall assume that
  $A\not\in \FF$ and derive a contradiction by building a loop in $\FF$ that encloses $A$
  and which can be continuously contracted into a point in $\FF$.

  Since $\lossinf(c) < \ell< \lossup(c)$, there exist $\etainf$ and $\etasup$ such that:
  \[
    C(\etainf) \leq c, \quad \loss(\etainf) < \ell, \quad C(\etasup) \geq c \quad
  \text{and}\quad \loss(\etasup) > \ell.
  \]
  We concatenate the four continuous paths in $\Delta$ defined for $u\in [0, 1]$:
  \[
    u \mapsto u\etainf, \quad
    u\mapsto (1-u) \etainf + u \un, \quad u\mapsto (1-u) \un + u\etasup \quad\text{and}\quad
  u\mapsto (1-u) \etasup,
  \]
to obtain a continuous loop $(\eta_t, t\in [0, 4])$ from $[0,4]$ to $\Delta$, such that:
  \[
    (\eta_0,\eta_1,\eta_2,\eta_3,\eta_4)
     = ( \zero,  \etainf,  \un,\etasup, \zero).
  \]

  We now define a continuous family of loops $(\gamma_s, s\in [0, 1])$ in $\R^2$ by
  \[
    \gamma_s(t) = (C(s \eta_t),\loss(s\eta_t), t\in [0, 4]).
  \]
  By definition, for all $s\in[0,1]$, $\gamma_s$ is a continuous loop
  in $\FF$. Since $A = (c,\ell) \notin \FF$, the loops~$\gamma_s$ do not
  contain~$A$, so the winding number $W(\gamma_s,A)$ is well-defined
  (see for example \cite[Definition 6.1]{HJ09}). As $A\not\in \FF$, we
  get that~$\gamma_s$ is a continuous deformation in
  $\R^2\setminus\{A\}$ from~$\gamma_1$ to~$\gamma_0$. Thanks to
  \cite[Theorem 6.5]{HJ09}, this implies that~$W(\gamma_s,A)$ does not
  depend on $s\in [0, 1]$.

  For $s=0$, the loop degenerates to the single point $(C(0),0)$ so the winding number is
  $0$. For $s=1$, let us check that the winding number is $1$, which will provide the
  contradiction. To do this, we compare $\gamma_1$ with a simpler loop $\delta$ defined
  by:
  \[ \delta (0) = \delta (4) = (\maxcost,0), \quad \delta (1) = (0,0), \quad \delta (2) =
    (0,\maxloss) \quad\text{and}\quad
    \delta (3) = (\maxcost,\maxloss),
  \]
  and by linear interpolation for non integer values of $t$: in other words, $\delta $ runs
  around the perimeter of the axis-aligned rectangle with corners $(0,0)$ and
  $(\maxcost,\maxloss)$. Clearly, we have $W(\delta , A)=1$.

  Let $M_t$, $N_t$ denote $\gamma_1(t)$ and $\delta (t)$ respectively. For $t\in [0,1]$,
  we have $N_t = ((1-t)\maxcost,0)$, so the second coordinate of $\overrightarrow{AN_t}$
  is non-positive. On the other hand $\loss(t\etainf) \leq \loss(\etainf) < \ell$, so the
  second coordinate of $\overrightarrow{AM_t}$ is negative. Therefore the two vectors
  $\overrightarrow{AN_t}$ and $\overrightarrow{AM_t}$ cannot point in opposite directions.
  Similar considerations for the other values of $t\in [1, 4]$ show that
  $\overrightarrow{AN_t}$ and $\overrightarrow{AM_t}$ never point in opposite directions.
  By \cite[Theorem 6.1]{HJ09}, the winding numbers $W(\gamma_1, A)$ and $W(\delta, A)$ are
  equal, and thus $W(\gamma_1, A)=1$.

  This gives that $A\in \FF$ by contradiction, and thus $F_1\subset \FF$.

  \medskip

  Finally, it is easy to check that $F_1$ has a connected complement, because $F_1$ is
  bounded, and all the points in $F_1^c$ can reach infinity by a straight line: for
  example, if $\ell> \lossup(c)$, then the half-line $\{(c,\ell'), \ell' \geq \ell\}$ is
  in $F_1^c$.
\end{proof}

\subsection{Geometric properties}
  When  the cost function is  affine, then
there is a nice geometric property of the Pareto frontier.

\begin{lemma}\label{lem:affine-chord}
  Suppose that Assumption \ref{hyp:loss+cost} holds, the cost function
  $C$ is affine, and the loss function $\loss$ is sub-homogeneous.
  Then, we have for all $c\in [0, \maxcost]$ and $\theta\in[0,1]$:
 \begin{equation}
   \label{eq:corde}
\lossinf( \theta c+ (1-\theta)\maxcost ) \leq \theta \lossinf(c).
 \end{equation}
\end{lemma}

\begin{remark}
 Geometrically, Lemma~\ref{lem:affine-chord} means that the graph of the loss
 $\lossinf \, \colon \, [0,\maxcost] \to [0,\maxloss]$ is below its chords with end point
 $(1,\lossinf(\un)) = (1,0)$. See Fig.~\ref{fig:pareto_frontier} for a typical
 representation of the Pareto frontier (red solid line).
\end{remark}

\begin{proof}
  Let $\theta\in[0,1]$.
  For $c\in [\Cinf(0), \maxcost]$, we have $\lossinf(c)=0$,
  so~\eqref{eq:corde} is trivially true. 

  Let $c\in [0, \Cinf(0)]$. 
Thanks to Lemma
 \ref{lem:c-dec+L-hom}, Assumption \ref{hyp:loss} holds. Thus, thanks
 to Propositions~\ref{prop:main-result}
 and~\ref{prop:f_properties}~\ref{prop:c_constraint_binding},
 there exists $\eta\in \cp$ with cost $C(\eta)=c$ and thus
 $\loss(\eta)=\lossinf (c)$. Since $C$ is affine, we have:
 \[
   C(\theta \eta) = \theta C(\eta)+ (1 - \theta) C(\zero) = \theta c
   +(1- \theta)\maxcost .
 \]
 Therefore,    the    strategy    $\theta\eta$   is    admissible    for
 Problem~\eqref{eq:Prob1}          with         cost          constraint
 $C(\cdot)  \leq \theta  c  + (1-\theta)  \maxcost$.  This implies  that
 $\lossinf(\theta c+  (1-\theta) \maxcost) \leq \loss(\theta  \eta) \leq
 \theta \loss(\eta)= \theta \lossinf(c)$,  thanks to the sub-homogeneity
 of the loss function $\loss$.
\end{proof}

In some cases, see for example \cite[Section 4]{ddz-Re}, it is possible to prove that the
loss function is a convex function  (which in turn implies
Assumption~\ref{hyp:loss}). In this case, choosing a convex cost function implies that
Assumption~\ref{hyp:cost} holds and the Pareto frontier is convex. A similar result holds
in the concave case. We provide a short proof of this result.

\begin{proposition}
  [Convexity/concavity of the Pareto frontier]%
  \label{prop:cvex}
  Suppose that Assumption \ref{hyp:loss+cost} holds. If the cost function $C$ and the loss
  function $\loss$ are convex, then the functions $\Cinf$ and $\lossinf$ are convex. If
  the cost function $C$ and the loss function $\loss$ are concave, then the functions
  $\Csup$ and $\lossup$ are  concave.
\end{proposition}

\begin{proof}
  Let $\ell_0,\ell_1 \in [0, \maxloss]$. By Proposition \ref{prop:main-result}, there
  exist $\eta_0, \eta_1\in \Delta$ such that $\loss (\eta_i) \leq \ell_i$ and $C(\eta_i) =
  \Cinf(\ell_i)$ for $i\in \{0,1\}$. For $\theta\in [0, 1]$, let $\ell=(1-\theta) \ell_0 +
  \theta \ell_1$. Since $C$ and $\loss$ are assumed to be convex, $\eta = (1-\theta)\eta_0
  + \theta \eta_1$ satisfies:
  \[
    C(\eta) \leq (1-\theta)\Cinf(\ell_0) + \theta
    \Cinf(\ell_1) \quad\text{and}\quad \loss (\eta) \leq (1-\theta)\ell_0 + \theta \ell_1.
  \]
  Therefore, we get that $\Cinf((1-\theta)\ell_0 + \theta \ell_1) \leq C(\eta) \leq
  (1-\theta)\Cinf(\ell_0) + \theta \Cinf(\ell_1)$, and $\Cinf$ is convex. The proof of the
  convexity of $\lossinf$ is similar. The concave case is also similar.
\end{proof}

\subsection{Stability}
In view of the application to vaccination optimization, the cost
function is known but the loss function depends on parameters that have
to be estimated, and thus the loss function might be approximated by a
sequence of loss functions, see~\cite{ddz-theory-topo} for the stability
of $R_e$ and $\I$. 
It is then natural to consider the stability of the Pareto frontier and the set of
Pareto optima. Recall that under Assumptions~\ref{hyp:loss+cost} and~\ref{hyp:loss},
 thanks to~\eqref{eq:FL=C*}, the graph
$\{(c,\lossinf(c)) \, \colon \, c \in [0,\maxcost]\}$ of~$\lossinf$ is
the union of the Pareto frontier and the straight line joining
$(0,\Cinf(0))$ to $(0,\maxcost)$ and can thus be seen as an extended
Pareto frontier. The  following proposition, whose proof  is immediate, 
implies in particular the convergence of the extended Pareto
frontier. This result can also easily be adapted to the anti-Pareto frontier.
To stress the dependence  in the loss function $\loss$, we shall write $
\mathcal{P}_\loss$ for  the set of
Pareto optima $\mathcal{P}$.

\begin{proposition}[Stability of the set of Pareto optima]
  \label{prop:F-stab}
  Let $C$ be a cost function  and $(\loss^{(n)}, n\in \N)$ a sequence of
  loss functions  converging uniformly  on $\Delta$  to a  loss function
  $\loss$. Assume  that Assumptions  \ref{hyp:loss+cost}, \ref{hyp:cost}
  and  \ref{hyp:loss} hold  for  the  cost $C$  and  the loss  functions
  $\loss^{(n)}$, $n\in \N$, and $\loss$. Then $\lossinf^{(n)}$ converges
  uniformly  to $\lossinf$.  Let  $\eta\in  \Delta$ be  the  limit of  a
  sequence   $(\eta_n,   n\in   \N)$   of   Pareto   optima,   that   is
  $\eta_n\in  \mathcal{P}_{\loss^   {(n)}}$  for   all  $n\in   \N$.  If
  $C(\eta)\leq \Cinf(0)$, then we have $\eta\in \mathcal{P}_{\loss}$.
\end{proposition}

\section{SIS model and vaccination strategies}
\label{sec:k-SIS-model}

We  follow  the presentation  of~\cite{delmas_infinite-dimensional_2020}
for an heterogeneous  SIS model in a constant size  population, and give
natural examples for the loss  function and cost function of vaccination
strategies.

\subsection{The setting}
Let~$(\Omega, \cf,  \mu)$ be a  probability space, where~$x  \in \Omega$ represents  a
\emph{feature}  and the  probability measure~$\mu(\mathrm{d}  x)$ represents the fraction of the
population with feature~$x$. The parameters of the SIS model are given by a \emph{recovery
rate function} $\gamma$, which is a positive bounded measurable function defined on
$\Omega$, and a \emph{transmission rate kernel} $k$, which is a non-negative measurable
function defined on $\Omega^2$. \medskip

In accordance with \cite{delmas_infinite-dimensional_2020}, we consider
for a kernel $\kk$
on $\Omega$ and $q\in (1, +\infty )$ its   norms:
\begin{equation*}
  \norm{\kk}_{\infty ,q}=\sup_{x\in \Omega}\, \left(\int_\Omega \kk(x,y)^q\,
    \mu(\mathrm{d} y)\right)^{1/q}
\end{equation*}
For a kernel $\kk$ on $\Omega$ such that $\norm{\kk}_{\infty ,q}$ is finite for some $q\in
(1, +\infty )$, we define the integral operator~$\Tinf_{\kk}$  on the set $\cl$ of bounded
measurable real-valued function on $\Omega$ by:
\begin{equation}\label{eq:def-Tkk}
  \Tinf_\kk (g) (x) = \int_\Omega \kk(x,y) g(y)\,\mu(\mathrm{d}y)
  \quad \text{for } g\in \cl \text{ and } x\in \Omega.
\end{equation}
When $\kk$ is non-negative, this operator is
also positive,  that is,  $\Tinf_\kk(\mathscr{L}^\infty_+)\subset
\mathscr{L}^\infty_+$, where $\mathscr{L}^\infty_+=\{f\in
\mathscr{L}^\infty\, \colon\, f\geq \zero\}$ is the cone of non-negative
bounded measurable functions.

We shall consider the kernel $\kkk=k/\gamma$   defined by:
\begin{equation}
  \label{eq:def-kk}
 {    \kkk(x,y)=k(x,y)\, \gamma(y)^{-1}.}
\end{equation}

\begin{hyp}[On the SIS model]\label{hyp:k-g}
  Let  $(\Omega,\cf,\mu)$ be  a  probability space.   The recovery  rate
  function~$\gamma$  defined  on  $\Omega$  is  measurable  bounded  and
  positive.    The  transmission   rate   kernel~$k$   on  $\Omega$   is
  non-negative         measurable          and         such         that
  $\norm{k/\gamma}_{\infty , q}<+\infty $ for some $q\in (1, +\infty )$.
  \end{hyp}

  Under  Assumption~\ref{hyp:k-g}, the  integral operator~$\Tinf_{\kkk}$
  on $\mathscr{L}^\infty$ is   a bounded positive linear operator
  on  $\mathscr{L}^\infty$; it  is the  so called  \emph{next-generation
    operator}.  Notice  also that  $\norm{k}_{\infty ,q}$ is  finite (as
  $\gamma$  is bounded),  and  thus  the operator  $\Tinf_k$  is also  a
  bounded positive linear operator on $\mathscr{L}^\infty$.

\subsection{The SIS model}
\label{sec:sis-model}

Let $\Delta'=\{f\in\mathscr{L}^\infty\, \colon\, 0\leq  f\leq  1\}$ be the
subset of non-negative functions bounded by~$1$.
   The  SIS dynamics  considered
in~\cite{delmas_infinite-dimensional_2020} follows  the vector field~$F$
defined on~$\Delta'$ by:
\begin{equation}\label{eq:vec-field}
 {  F(g) = (1 - g) \Tinf_k (g) - \gamma g,}
\end{equation}
where the term $(1 -g)$ corresponds to the \emph{incidence rate} given
by the usual \emph{law of mass action}.

More precisely, let~$u_t(x)\in [0, 1]$ models the proportion of 
individuals with feature~$x\in \Omega$ which are  infected at time~$t\geq 0$.
 The SIS dynamics considered for~$u=(u_t, t\in \R)$  
 is given by the following EDO in $\mathscr{L}^\infty$:
\begin{equation}\label{eq:SIS2}
  {   \partial_t u_t = F(u_t)} \quad\text{for } t\in \R_+,
\end{equation}
with initial condition~$u_0\in \Delta'$. It is proved
in~\cite{delmas_infinite-dimensional_2020} that such a solution~$u$
exists and is unique.

\medskip

An \emph{equilibrium}  of~\eqref{eq:SIS2} is a function~$g  \in \Delta'$
such that~$F(g) =  0$.  Notice that the  \emph{disease free equilibrium}
(DFE), which is the zero  function, is indeed an equilibrium.  According
to  \cite{delmas_infinite-dimensional_2020},  there   exists  a  maximal
equilibrium~$\mathfrak{g}$, \textit{i.e.}, an  equilibrium such that all
other   equilibria~$h\in   \Delta'$  are   dominated   by~$\mathfrak{g}$:
$h \leq \mathfrak{g}$.  It is  towards this maximal equilibrium that the
process  stabilizes  when  started  from   a  situation  where  all  the
population is infected, that is, $ {  \lim_{t\rightarrow \infty }
  u_t=\mathfrak{g}} $ if $u_0(x)=1$ for all $x\in \Omega$.

For $T$ a bounded operator on $\cl$ endowed with its usual  supremum
norm, we define by~$\norm{T}_{\cl}$ its operator norm and its spectral
radius is given by $
   \rho(T)=  \lim_{n\rightarrow \infty } \norm{T^n}_{\cl}^{1/n}$.
The \emph{reproduction  number}~$R_0$ associated to the  SIS model given
by~\eqref{eq:SIS2}  is  the  spectral   radius  of  the  next-generation
operator:
\begin{equation}\label{eq:def-R0-2}
  {   R_0= \rho (\Tinf_{\kkk}).}
\end{equation}
If~$R_0\leq 1$  (sub-critical and  critical case)
the maximal
equilibrium is the DFE and furthermore~$u_t$ converges
pointwise  to zero  when~$t\to\infty$.   If~$R_0>1$
(super-critical  case), the  maximal equilibrium  $\mathfrak{g}$ is  not
the  disease free equilibrium and~$\int_\Omega \mathfrak{g} \, \mathrm{d}\mu > 0$.

\medskip

We will denote by~$\I_0 $ the \emph{fraction of infected individuals at
the maximal  equilibrium}:
\begin{equation}\label{eq:def-I0-2}
   {\I_0=\int_\Omega \mathfrak{g}
  \,  \rd \mu,}
\end{equation}
where we write $\int_\Omega f \, \rd \mu$ for $\int_\Omega f(x) \, \mu(\rd
x)$. 
We deduce that:
\[
  \I_0>0\quad \Longleftrightarrow\quad   R_0>1.
\]

\subsection{Vaccination strategies}\label{sec:vacc}
We  assume the  vaccine has  \emph{perfect  efficiency} and  that it  is
applied  once   and  for  all   at  time  $t=0$.    A  \emph{vaccination
  strategy}~$\eta$   is   an   element   of~$\Delta'$,   where~$\eta(x)$
represents the proportion  of \emph{\textbf{non-vaccinated}} individuals
with feature~$x$.  Notice that~$\eta\, \mathrm{d} \mu$  corresponds in a
sense to the effective population.  In particular, the ``strategy'' that
consists  in  vaccinating  no  one  (resp.   everybody)  corresponds  to
$\eta(x) =1$ (resp. $\eta(x) =0$) for all $x\in \Omega$.

\medskip

   For~$\eta       \in      \Delta'$,      the
kernel~$\kkk\eta=k\eta/\gamma$, defined by $\kkk(x,y)=k(x,y)
\eta(y)/\gamma(y)$,           has         finite          norm
$\norm{\cdot}_{\infty ,  q}$, so  we can  consider the  bounded positive
operators~$\Tinf_{\kkk        \eta         }$        and~$\Tinf_{k\eta}$
on~$\mathscr{L}^\infty$.                   According                  to
\cite[Section~5.3.]{delmas_infinite-dimensional_2020}, the  SIS equation
with vaccination strategy~$\eta$  is given by 
$u^\eta=(u^\eta_t, t\geq 0)$ solution to~\eqref{eq:SIS2} with~$F$
is replaced by~$F_\eta$ defined by:
\begin{equation}\label{eq:vec-field-vaccin}
  {   F_\eta(g) = (\un - g) \Tinf_{k\eta}(g) - \gamma g.}
\end{equation}
Then the quantity~$u_t^\eta(x)=u^\eta(t,x)$  represents  the  probability  for  a  non-vaccinated
individual  of feature~$x$  to be  infected at  time $t$;  so at time
$t$ among  the
population  of feature~$x$,  a  fraction $1-\eta(x)$  is  vaccinated,  a
fraction $\eta(x)\, u_t^\eta(x)$ is  not vaccinated and infected, and a
fraction $\eta(x)\, (1-u_t^\eta(x))$ is  not vaccinated and
not infected.

We define the \emph{effective reproduction number} $R_e(\eta)$
associated to the vaccination strategy $\eta$ as  the spectral radius
of the effective next-generation operator~$\Tinf_{\kkk \eta}$:
\begin{equation}\label{eq:def-R_e}
   {R_e(\eta)=\rho(\Tinf_{\kkk\eta}).}
\end{equation}
We denote by~$\mathfrak{g}_\eta$  the corresponding maximal equilibrium.
We will  denote by~$\I(\eta)  $ the  corresponding fraction  of infected
individuals at equilibrium. Since the probability for an individual with
feature~$x$    to    be    infected    in    the    stationary    regime
is~$\mathfrak{g}_\eta(x)  \, \eta(x)$,  this  fraction is  given by  the
following formula:
\begin{equation}\label{eq:asymptotic_number_endemic}
 {  \I (\eta)=\int_\Omega \mathfrak{g}_\eta
  \, \eta\, \mathrm{d}\mu.}
\end{equation}

As           $F_\eta(\mathfrak{g}_\eta)=0$,          we           deduce
that~$\mathfrak{g}_\eta\eta=0$   $\mu$-almost   surely   if   and   only
if~$\mathfrak{g}_\eta=\zero$.
Applying        the        results
of the previous section  to the  kernel~$k \eta$,  we
deduce that:
\begin{equation}
  \label{eq:gh>0}
  \I (\eta)>0 \,\Longleftrightarrow\, R_e(\eta)>1.
\end{equation}

\subsection{The set of vaccination strategies $\Delta$}
\label{sec:topo-D}
This section is motivated by the continuity results
from~\cite{ddz-theory-topo} established   for   the two  loss functions
$R_e$ and  $\I$. For~$p    \in     [1,    +\infty]$,    
the space  $L^p$   of real-valued  measurable
functions~$g$            defined~$\Omega$           such            that
$\norm{g}_p=\left(\int |g|^p  \, \mathrm{d} \mu\right)^{1/p}$  (with the
convention  that~$\norm{g}_\infty$  is the~$\mu$-essential  supremum  of
$|g|$) is  finite, where  functions which agree~$\mu$-almost  surely are
identified.
We 
denote  by $\zero$  and  $\un$ the  elements of  $L^p  $ which  are
respectively the (class of equivalence  of the) constant functions equal
to     $0$     and     to     $1$. 
  We     consider     the
  cone~$L^p_+=\{f\in L^p \, \colon\, f\geq \zero\}$ of~$L^p$; it is 
 salient pointed convex and corresponds to the 
usual partial  order $\leq  $, where $f\leq g$ means
that $\mu(f>g)=0$.

We will  mainly be interested in  the case $p=+\infty $.   With a slight
abuse   of   language   we   shall    identify   a   function   $f$   in
$\mathscr{L}^\infty$  with  its equivalent  class  in  $L^\infty $.   We
consider the set of vaccination strategies:
\begin{equation}
   \label{eq:def-D}
  \Delta=\DeltaSIS = \{f\in L^\infty\,  \colon\, \zero\leq  f\leq \un\}.
\end{equation}
  In the following we only consider this choice for
  $\Delta$, and drop the subscript SIS
  for brevity.

We now endow $L^\infty $ with  a topology for which $\Delta$ is compact.
Since the  measure $\mu$  is finite,  the topological  dual of  $L^1$ is
$L^\infty $.  Let $\co$ denote  the weak-* topology on~$L^\infty $, that
is, the weakest  topology on $L^\infty $ for which  all the linear forms
$f\mapsto \int_\Omega fg \, \rd \mu$, $g\in L^1$, defined on $L^\infty $
are  continuous. following~\cite{brezis2010functional},  we recall  that
$(L^\infty  ,  \co)$  is  an Hausdorff  topological  vector  space,
the topological dual
of   $(L^\infty  ,   \co)$  is   $L^1$, 
and   a
sequence~$(g_n,  \, n  \in  \N)$ of  elements  of~$L^\infty $  converges
weakly-*   to~$g\in   L^\infty   $   if   and   only   if:
\begin{equation}\label{eq:weak-cv}
  \lim\limits_{n \to \infty} \int_\Omega  g_n\, h \, \mathrm{d}\mu= \int_\Omega 
  g\, h\, \mathrm{d}\mu \quad\text{for all~$h \in L^1 $.}
\end{equation}
 According  to Alaoglu's  theorem, the  set $\Delta$  is compact  for the
 weak-*     topology.

\subsection{The cost function and the loss functions for the SIS model}
\label{sec:SIScostandloss}
To  any vaccination  strategy $\eta  \in  \Delta$, we  associate a  cost
$C(\eta)$  which measures  all  the costs  of  the vaccination  strategy
(production  and diffusion).  The cost  is expected  to be  a non-increasing
function of~$\eta$, since~$\eta$  encodes the non-vaccinated population.
Since doing nothing costs nothing, we also expect $C(\un)=0$.
   A simple  cost model is the
    affine cost given by:
    \begin{equation}\label{eq:def-C0}
      \costa(\eta)=\int_\Omega (\un-\eta)\, \costad \,
    \rd  \mu,
    \end{equation}
    where $\costad(x)$ is the cost  of vaccinating population of feature
    $x\in \Omega$, with $\costad\in  L^1_+$ and $\costad\neq \zero$. 
    The  particular   case
    $\costad=\un$  is    the   uniform
    cost~$C = \costu$:
    \begin{equation}\label{eq:def-C}
      \costu (\eta) = \int_\Omega (1 - \eta) \,
      \mathrm{d}\mu.
    \end{equation}
    The real cost of the vaccination  may be a more complicated function
    $\psi(\costa(\eta))$ of the affine cost, for example if the marginal
    cost  of  producing  a  vaccine  depends  on  the  quantity  already
    produced. However,  as long as  $\psi$ is strictly  increasing, this
    will not affect the optimal strategies.

\medskip

We now consider the loss  $\loss(\eta)$  which measures  the (non)-efficiency  of  the  vaccination
strategy~$\eta$. Different choices are possible, we shall concentrate on
the effective reproduction number $R_e$ and the asymptotic proportion of
infected individuals $\I$.

We say a function $H$ defined on $\Delta'$ is \emph{well defined} on $\Delta$
if   $H(\eta)=H(\eta')$ for
all $\eta, \eta'\in \Delta'$ such that $\eta=\eta'$ $\mu$-almost surely;
with a slight abuse of notation, we still denote  by $H$
the corresponding  function defined on $\Delta$.
 The     following    results     are     proved
in~\cite{ddz-theory-topo}.

\begin{proposition}[Properties of $R_e$ and $\I$]\label{prop:R_e}
  Suppose Assumption \ref{hyp:k-g} holds. 
  \begin{enumerate}[(i)]
  \item\label{prop:a.s.+increase-Re}%
    The  functions  $R_e$ and $\I$ are  well  defined  on  $\Delta$;
    they are   also
    non-decreasing, sub-homogeneous and  continuous on $\Delta$ (endowed
    with the weak-* topology).
  \item\label{prop:min_Re}%
   $R_e(\un) = R_0$ and~$R_e(\zero) = 0$.
 \item\label{prop:min-I} $\I(\un)=\I_0$ and, for $\eta\in \Delta$, we
   have  $\I (\eta)=0$  if
    and only if $R_e(\eta) \leq 1$.
  \end{enumerate}
\end{proposition}

\begin{remark}[On the generality of the uniform cost]
  \label{rem:costa-costu}
  For the  reproduction number  optimization in the  vaccination context
  (that is,  $\loss=R_e$), one can  without loss of  generality consider
  the  uniform cost instead  of  the affine  cost.
  Indeed, consider  the SIS model with parameters  $\param=[(\Omega,
  \cf, \mu), k, \gamma]$
  with  the affine  cost function  $\costa$ and  the loss  $R_e$. If  we
  assume furthermore that  $\costad$ is bounded and bounded  away from 0
  (that is $\costad$  and $1/\costad$ belongs to  $L^\infty _+$), and, without
  loss of generality,  that $\int_\Omega \costad \, \mathrm{d}  \mu=1$, then we
  can       consider        the       weighted        kernel       model
  $\param_0=[(\Omega,   \cf,  \mu_0),   k_0, \gamma]$   with
  probability   measure
  $\mu_0(\mathrm{d}  x)=\costad(x)  \,  \mu(\mathrm{d}  x)$  and  kernel
  $k_0= k/\costad$.  (Notice  that if Assumption~\ref{hyp:k-g} holds
  for the model $\param$, then it  also holds for the model $\param_0$.)
  Consider  the  loss  $\loss=R_e$.   Then, using  that  $L^p(\mu)$  and
  $L^p(\mu_0)$ are  compatible and the corresponding  integral operators
  are   consistent   (see~\cite[Section~2.2]{ddz-Re}),   we   get   that
  $(\costa(\eta),  \loss (\eta))$  for the  model $\param$  is equal  to
  $(\costu(\eta),  \loss  (\eta))$ for  the  model  $\param_0$, for  all
  strategies $\eta\in \Delta$.

  Therefore, for  the loss function  $\loss=R_e$, instead of  the affine
  cost $\costa$,  one can consider  without any real loss  of generality
  the uniform cost.  However, this is  no longer the
  case for the loss function $\loss=\I$.
\end{remark}

\subsection{Irreducible and monatomic kernels}
\label{sec:prop-L-atomic}

We  follow the  presentation in~\cite[Section~5]{ddz-Re}  on the  atomic
decomposition of  positive compact operator  and Remark 5.2  therein for
the  particular case  of  integral operators,  see  also the  references
therein for further  results.  For $A, B\in \cf$, we  write $A\subset B$
a.s.\ if $\mu(B^c  \cap A)=0$ and $A=B$ a.s.\ if  $A\subset B$ a.s.\ and
$B\subset A$ a.s..   If $\cg\subset \cf$ is a  $\sigma$-field, we recall
that $A\in \cg$ is  an atom of $\mu$ in $\cg$ if  $\mu(A)>0$ and for any
$B\subset A$  we have a.s.\  either $B=\emptyset$ or $B=A$.  Notice that
the atoms are defined up to an a.s.\ equivalence.

Let $\kk$ be a kernel on $\Omega$ such that
$\norm{\kk}_{\infty ,q}<+\infty $ for some $q\in (1, +\infty )$.  For  $A, B\in  \cf$, we  simply write:
\[
  \kk(B, A)=
  \int_{B \times A} \kk(z,y)\, \mu(\rd z) \mu(\rd y).
\]
A   set  $A\in   \cf$  is   called  \emph{$\kk$-invariant},   or  simply
\emph{invariant}  when there  is no  ambiguity on  the kernel  $\kk$, if
$\kk(A^c,  A)=0$.   In  the  epidemiological setting with $\kk$ the
transmission rate kernel,  the  set  $A$  is
invariant if the  sub-population $A$ does not  infect the sub-population
$A^c$.  The kernel $\kk$  is \emph{irreducible} (or \emph{connected}) if
any invariant set $A$ is  equal a.s.\ either to $\emptyset$ or to 
$\Omega$.  If $\kk$
is irreducible,  then either $\rho(\Tinf_\kk)>0$
or $\kk\equiv 0$ and $\Omega$ is
an  atom  of $\mu$  in  $\cf$  (degenerate  case). A  simple  sufficient
condition for irreducibility is for the kernel to be a.s.\ positive.

\medskip

Let $\ca$ be  the set of $\kk$-invariant sets. Let  us stress that $\ca$
depends only on  the support of the kernel $\kk$.   Notice that $\ca$ is
stable   by  countable   unions   and   countable  intersections.    Let
$\cfi=\sigma(\ca)$ be the $\sigma$-field  generated by $\ca$.  Then, the
kernel $\kk$  restricted to an atom  of $\mu$ in $\cfi$  is irreducible.
We shall only consider non degenerate  atoms, and say the atom (of $\mu$
in $\cfi$)  is non-zero if the  restriction of the kernel  $\kk$ to this
atom  is non-zero  (and thus  the spectral  radius of  the corresponding
integral  operator   is  positive).    We  say   the  kernel   $\kk$  is
\emph{monatomic} if there exists a  unique non-zero atom, say $\oa$, and
the  kernel   is  \emph{quasi-irreducible}   if  it  is   monatomic  and
$\kk\equiv 0$ outside $\oa\times \oa$, where $\oa$ is its non-zero atom.
Notice that if $\kk$ is irreducible with $\rho(\Tinf_\kk)>0$, then $\kk$
is monatomic with non-zero atom $\oa=\Omega$ a.s.. The quasi-irreducible
property  is the  usual extension  of  the irreducible  property in  the
setting of symmetric kernels; and  the monatomic property is the natural
generalization        to        non-symmetric        kernels,        see
also~\cite{dlz-atom,dlz-equilibre}   for   other  characterizations   of
monatomic kernels.

\begin{remark}[Epidemiological interpretation]\label{rem:lien-epidemie}
   When the transmission rate  kernel $k$ for   the    SIS   model
   from Section~\ref{sec:sis-model}
   is  monatomic, with
   non-zero atom $\oa$,  then the population with trait in  $\oa$ can infect
   itself. It may also infect another part of the population, say with
   trait in $\oi$, but:
   \begin{itemize}
   \item the infection cannot be sustained at all in $\oi$: $k$ is
     quasi-nilpotent on $\oi$;
     \item the population with trait in $\oi$ does not infect
       back the non-zero atom $\oa$.  
     \end{itemize}
Recall the reproduction  number $R_0$ defined in~\eqref{eq:def-R0-2}.       If  furthermore  $R_0>1$, then  the  set  $\oa\cup\oi$
   corresponds  to  the  support  of  the  maximal  endemic
   equilibrium, which is also the unique non-zero equilibrium,
   see~\cite{ddz-theory-topo}. 
\end{remark}

\subsection{Properties of the cost and loss functions}
\label{sec:prop-C+L}

We shall now consider the natural loss functions $R_e$ and $\I$.
Recall  $\Delta$ defined by~\eqref{eq:def-D}
is endowed   with  the   weak-*
topology.

\begin{proposition}[Loss function properties for the SIS model]
  \label{prop:prop-loss}
  Suppose Assumption~\ref{hyp:k-g}  on the SIS model  holds and consider
  the loss function $\loss\in \{R_e, \I\}$ on $\Delta$.
  \begin{enumerate}[(i)]
  \item \label{it:thm-loss+cost}  The loss function is non-decreasing and
    continuous           with          $\loss(\zero)=0$.            Thus
    Assumption~\ref{hyp:loss+cost}  on  the  loss  holds  for  the  loss
    $\loss=R_e$ provided  that $R_0>0$ (and $\maxloss=R_0>0$)  and for the
    loss     $\loss=\I$    provided     that     $R_0>1$    (and 
    $\maxloss=\I_0>0$).
\item \label{it:thm-loss} Any local minimum of the loss function $\loss$ is global, that
  is, Assumption~\ref{hyp:loss}  holds. 
\item \label{it:thm-loss**}
  If the
  transmission rate kernel  is  monatomic,   then
$R_0>0$ and any  local maximum of the loss function $\loss=R_e$ is
global (that is, Assumption~\ref{hyp:loss**} holds for $\loss=R_e$).
If furthermore $R_0>1$,  then
  any local maximum of the loss function $\loss=\I$, which is not a zero
  of the loss, is global (that is, Assumption~\ref{hyp:loss=I} holds for $\loss=\I$).
  \end{enumerate} 
\end{proposition}

\begin{proof}
  We  deduce  from  Proposition~\ref{prop:R_e} that  the  loss  function
  $\loss=R_e$           (resp.           $\loss=\I)$           satisfies
  Assumption~\ref{hyp:loss+cost}, provided that $R_0>0$ (resp.  $R_0>1$)
  so that $\maxloss>0$. This gives Point~\ref{it:thm-loss+cost}. 
  Then  use Lemma~\ref{lem:c-dec+L-hom}
  and    Proposition~\ref{prop:R_e}   again    to   get
Point~\ref{it:thm-loss}, that is, 
Assumption~\ref{hyp:loss} holds.
Point~\ref{it:thm-loss**} is in~\cite[Lemmas~5.4 and~5.5]{ddz-theory-topo}. 
\end{proof}

\begin{remark}[Local maxima of the loss $\I$]
  \label{rem:local-max-I}
  The    loss     $\loss=\I$    does    not    satisfy
Assumption~\ref{hyp:loss**}  in general  even   when    the   kernel    $k$   is
irreducible with $R_0>0$.  Indeed, by continuity of $R_e$, there exists a
(weak-*) open neighborhood~$O$ of $\zero$ in $\Delta$ such that $R_e(\eta) < 1$ for
all $\eta\in O$: consequently $\I$ is identically zero
on~$O$, and any $\eta\in O$ is a local maximum of
$\loss=\I$. However, these maxima are not global if~$\I(\un)>0$ (that
is,  $R_0>1$).
\end{remark}

We  give further  results  on the  ``worst''  vaccination strategy  that
changes nothing.  Those results are  stated assuming the  transmission rate
kernel  is monatomic.   The  general case  is more  delicate  and it  is
studied  in more  details in~\cite[Section~5]{ddz-cordon}  for the  loss
function $\loss=R_e$; in  particular Assumption~\ref{hyp:loss**} may not
hold and  the anti-Pareto  frontier may not  be connected.

Recall that
$\cplus=\Csup(0+)$, see~\eqref{eq:def-c+}.

\begin{proposition}[On the ``worst'' vaccination strategy]
\label{prop:L**}
Suppose  Assumption~\ref{hyp:k-g}  on  the  SIS model  holds,  that  the
transmission rate kernel is monatomic (and thus $R_0>0)$ with atom $\oa$
and  maximal equilibrium  $\mathfrak{g}$.   Suppose also  that the  cost
function satisfies Assumption~\ref{hyp:loss+cost}.
\begin{enumerate}[(i)]
\item\label{it:irr}
  If  $k$ is irreducible (with $R_0>0$), then we have  $\Csup
     (\maxloss)=0$ for the loss $\loss=R_e$ and also for the loss $\loss=\I$
provided $R_0>1$. 
   \item\label{it:oa}     
     For the  loss $\loss=R_e$, the strategy~$\ind{\oa}$, which vaccinates everyone
     except the atom,  is anti-Pareto optimal, with cost  $C(\ind{\oa})=\Csup
     (\maxloss)$. 
   \item \label{it:g} If $R_0>1$, for the  loss $\loss=\I$,
     the strategy
     $\ind{\{\mathfrak{g}>0\}}$, which consists in
vaccinating only the types
     that would not be infected at equilibrium, 
     is      anti-Pareto   optimal with
     $\Csup (\maxloss)=C(\ind{\{ \mathfrak{g}>0\}})$. 
     Furthermore, if the cost
  function $C$ is decreasing, then we have $\cplus<\maxcost$.
\end{enumerate}
\end{proposition}

According    to    Remark~\ref{rem:lien-epidemie}, if $R_0>1$,      we    have    that
$\ind{\oa}\leq  \ind{\{\mathfrak{g}>0\}}$,  with   an  equality  if  the
transmission rate kernel is quasi-irreducible.

\begin{proof}
  According       to~\cite[Lemma~5.4.(iii)]{ddz-theory-topo},       Assumption
  \ref{hyp:loss**} holds, $R_0>0$, and $\eta$ is a global maximum if and
  only if $\eta\geq \ind{\oa}$. 
  This directly implies that $ \ind{\oa}$
  is a global maximum of $R_e$. Assuming that the cost function satisfies
  Assumption~\ref{hyp:loss+cost}, we also get that $\Csup (\maxloss)=C(\ind{\oa})$ and
  also that $ \ind{\oa}$ is anti-Pareto optimal.  This gives
  Point~\ref{it:oa}.

  According to~\cite[Lemma~5.5.(iii)]{ddz-theory-topo}, if  $R_0>1$, and
  thus $\I(\un)>0$, then  Assumption~\ref{hyp:loss=I}  holds and $\eta$  is a
  global maximum  if and  only if $\eta\geq  \ind{\{ \mathfrak{g}>0\}}$.
  Arguing as  above, we get  the first part of  Point~\ref{it:g}.
  Thanks to Lemmas~\ref{lem:c0<cmax} and Remark~\ref{rem:local-max-I}, we
  get that 
  $\cplus < \maxcost$.   This ends the proof of  Point~\ref{it:g}.

  To  get Point~\ref{it:irr},  notice  that when  the transmission  rate
  kernel is  irreducible and $R_0>0$  then a.s.\ $\oa=\Omega$  and thus,
  thanks to  Remark~\ref{rem:lien-epidemie}, $\ind{\oa}= \un$.  Then use
  that  $C(\un)=0$  as  well  as  Points~\ref{it:oa}  and~\ref{it:g}  to
  conclude for the  loss $\loss=R_e$. The proof for  the loss $\loss=\I$
  is similar using   that
$\ind{\oa}\leq  \ind{\{\mathfrak{g}>0\}}$  when $R_0>1$. 
\end{proof}

The following result is an immediate consequence of
Lemmas~\ref{lem:c-dec+L-hom} and~\ref{lem:c-dec+L-hom**}.

\begin{corollary}[The affine cost function]
  \label{cor:prop-aff}
  The affine  cost function $\costa$ from~\eqref{eq:def-C0}  is continuous
on  $\Delta$  (endowed with  the  weak-*  topology) and  thus  satisfies
Assumptions~\ref{hyp:loss+cost},   \ref{hyp:cost}  and~\ref{hyp:cost**}.

If
furthermore $\costad$  is positive  (which is the  case for  the uniform
cost function $\costu$), then the cost is also decreasing and thus,  under
Assumption~\ref{hyp:k-g}  on  the  SIS model, 
$\lossinf(0)=\maxloss$  and $\lossup(\maxcost)=0$ 
for the loss $\loss=R_e$ provided $R_0>0$ and for the loss $\loss=\I$
provided $R_0>1$. 
\end{corollary}

In      view       of      Lemma~\ref{lem:5-et-5'}~\ref{it:5'=}      and
Remark~\ref{rem:5-et-5'-SIS} on the links between Assumptions~\ref{hyp:loss**}
and~\ref{hyp:loss=I}, we state the following technical lemma. 

\begin{lemma}[On the continuity of $\eta  \mapsto \eta  \vee \eta'$]
  \label{lem:cont}
  Let $\eta'\in \Delta$. The function $G(\eta)= \eta  \vee
  \eta'$ defined on $\Delta$ is continuous at $\eta=\zero$ (for the
  weak-* topology on $\Delta$). 
\end{lemma}

\begin{proof}
  According to \cite[Lemma~2.3]{ddz-theory-topo},  a function defined on
  $\Delta$ is continuous  if and only if it  is sequentially continuous,
  as  $\Delta$ (endowed with the weak-* topology) is a sequential
  space. Hence any.   Let $(\eta_n, n\in \N)$ be a  sequence of elements
  of    $\Delta$   which    weakly-*    converges    to   $\zero$.    As
  $f    \vee   g    =    f+g    -f   \wedge    g$,    to   prove    that
  $\lim_{n\rightarrow  \infty  }  G(\eta_n)=  G(\zero)=\eta'$  (for  the
  weak-* topology), it is enough to  prove that for all $h\in L^1_+$, we
  have:
\[
  \lim_{n \rightarrow \infty } \int _\Omega (\eta_n \wedge \eta')\, h \,
  \rd \mu= 0.
\]
This                 is                  obvious                 because
$ \int  _\Omega (\eta_n \wedge  \eta')\, h  \, \rd \mu\leq  \int _\Omega
\eta_n  \, h  \, \rd  \mu$  as $\eta_n$  and $h$  are non-negative,  and
because
$ \lim_{n \rightarrow \infty  } \int _\Omega \eta_n \, h  \, \rd \mu= 0$
as the sequence $(\eta_n, n\in \N)$ weakly-* converges to $\zero$.
\end{proof}

\subsection{Summary of the results for the SIS model}
\label{sec:summary}

We summarize in Table~\ref{tab:not-PAP} the results obtained when applying the
results of Section~\ref{sec:P-AntiP}  to  the  SIS  model  with  the  loss
function:
\[
  \loss\in \{R_e, \I\}
  \quad\text{and}\quad
  \Delta=\DeltaSIS=\{\eta\in L^\infty \, \colon\, \zero \leq  \eta \leq
  \un\}.
\]

  Suppose Assumption~\ref{hyp:k-g}  on the SIS model  holds and consider
  the loss function $\loss\in \{R_e, \I\}$ on $\Delta$.
We  assume that $C$ is  continuous with respect to the weak-* topology
   on $\Delta$ with $C(\un)=0$. To  avoid  technicalities,  we   focus  on  a
  decreasing cost function $C$ (notice this  is the case for the uniform
  cost $\costu$  in~\eqref{eq:def-C}, and more generally  for the affine
  cost $\costa$  in~\eqref{eq:def-C0} with positive  density $\costad$).
Finally, we assume that $R_0>0$ if $\loss=R_e$
   and $R_0>1$ if $\loss=\I$.

   We gather in Table~\ref{tab:not-PAP} the following:
   \begin{enumerate}
     \item The various notation and definitions.
     \item Properties of the Pareto frontier and associated functions, in
       the left column. Most of those follow from Proposition~\ref{prop:f_properties-2}. Proposition ~\ref{prop:f_properties} is used to show that $\Cinf$ decreases,
       and  Lemma~\ref{lem:c-dec+L-hom} to get the value of $\lossinf(0)$. 
     \item Properties of the feasible region, see Proposition~\ref{prop:trou}.
     \item Properties of the anti-Pareto frontier and associated functions,
       in the right column. For simplicity, we additionally assume
       here that the kernel is monatomic  and refer
         to~\cite[Section~5]{ddz-cordon} when this assumption does not
         hold. 
       Most results in this column come from
       Proposition~\ref{prop:f_properties**}
       (for $\loss=R_e$) and~\ref{prop:f_properties=I} (for $\loss=\I$).
       The value of $\Csup(\maxloss)$ is studied in
       Proposition~\ref{prop:L**}.
     
     \end{enumerate}

\begin{figure}
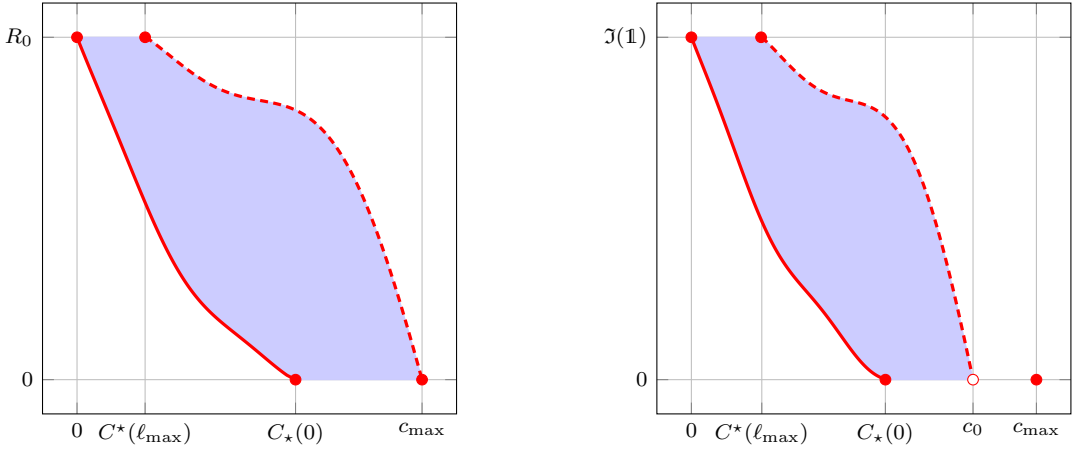

 \begin{subfigure}[T]{.5\textwidth} \centering
   \includestandalone{frontiers_re_mono}%
   \caption{%
     Loss function $\loss=R_e$   with $R_0>0$ (also $\maxloss=R_0$ and $\Csup(\maxloss)=C(\un_{\oa})$).
   }
   \label{fig:mona-Re}
 \end{subfigure}%
 \begin{subfigure}[T]{.5\textwidth}
   \centering
   \includestandalone{frontiers_i_mono}%
   \caption{%
   Loss  function
     $\loss=\I$ with $R_0>1$ (also $\maxloss=\I_0=\I(\un)>0$ and
     $\Csup(\maxloss)=C(\un_{\{\mathfrak{g}>0\}})$).
   }%
   \label{fig:mona-I}
 \end{subfigure}%
 \caption{%
   Typical shape of the Pareto frontier $\F$ in solid red line, the
   anti-Pareto frontier $\AF$ in red dashed line and the feasible
   region $\FF$ in light blue for the SIS model with a monatomic kernel.
 }%
 \label{fig:monatomic}
\end{figure}

     We also
represent  in  Fig.~\ref{fig:monatomic}  the  typical  Pareto  and
anti-Pareto  frontiers for  the loss  $R_e$ and  $\I$ with  a continuous
decreasing  cost function,  when  the transmission  rate  kernel $k$  is
monatomic.

\newcommand{\mc}[1]{\multicolumn{2}{l}{
    \hfill #1\hfill\hfill}}
\newcommand{\centeredcell}[1]{\multicolumn{1}{c}{#1}}
\newcommand{\bothlosses}[1]{
  #1
}
\newcommand{\maxcostrange}{\mathrm{range(\Csup)}}
\newcommand{\mincostrange}{\mathrm{range(\Cinf)}}
\begin{table}
   \centering
   \small
  \begin{tabular}{ @{\extracolsep{\fill}} lll}
    \toprule
    & {``Best'' vaccinations} &
     {``Worst'' vaccinations (monatomic $k$)} \\
    \midrule
    Assumptions & 
    \mc{Assumption~\ref{hyp:k-g} on SIS model, 
   $C$ continuous
                  decreasing.} \\
    \midrule
    Opt. problem  &
     $\min_\Delta(C,\loss)$ &
     \bothlosses{$\max_\Delta(C,\loss)$} \\
    \midrule
    Range of (cost,loss)
    &\mc{
     \( \begin{aligned}
         \maxcost & :\,=\max_\Delta C = C(\zero)>0 \\
         \maxloss & :\,=\max_\Delta \loss =    \loss(\un)>0 \\
         \sq & :\,=[0, \maxcost]\times [0, \maxloss]
      \end{aligned}\)
    }\\
    \midrule
    Opt.\ cost
     &
            $\Cinf (\ell):\,=\min_{\, \loss\leq  \ell}\,  C$
    &\bothlosses{ $\Csup (\ell):\, =\max_{\, \loss\geq  \ell}\,  C$}\\
    Opt.\ loss
    & $\lossinf (c):\,=\min_{\, C\leq  c}\,  \loss$
    &\bothlosses{ $\lossup (c):\, =\max_{\, C\geq  c}\,  \loss$}\\
    \midrule
      Possible outcomes
    & \mc{%
      \(\begin{aligned}
           \FF := (C, \loss)(\Delta) 
                &= \set{(c, \ell) \in \sq\, \colon\,
      \lossinf(c)\leq \ell \leq \lossup(c)} \\
                &= \set{(c, \ell) \in \sq\, \colon\, \Cinf(\ell)\leq c \leq  \Csup(\ell)} \\
          \FF \text{ is compact} & \text{ and path connected} \\
         \FF \text{ has no hole} & \text{ (that is, } \FF^c \text{ is connected}) 
         \end{aligned}
      \)
      } \\
    \midrule
    Opt.\ cost: properties
    &  $\Cinf$ is  decreasing
                            & \bothlosses{$\Csup$ is  decreasing}\\
    & \bothlosses{$\mincostrange: = \Cinf([0,\maxloss])$}
                                           & \bothlosses{$\maxcostrange: = \Csup([0,\maxloss])$} 
    \\
 & $\Cinf(0)\leq  \maxcost$
                            &$\Csup(0)=\maxcost$ and $\cplus := \Csup(0+)$
    \\
    & $\Cinf$ is continuous &
                               $\Csup$ is continuous on $(0, \maxloss]$ \\
    & &$\Csup$ is \(\left\{ \begin{aligned}
          &\quad\text{ continuous at $0$,  }   \cplus =  \maxcost & (\loss=R_e) \\
          &\text{discontinuous at $0$, } \cplus  < \maxcost & (\loss=\I)
        \end{aligned}\right.\) \\
   &  $\Cinf(\maxloss)=0$
                            &  $\Csup(\maxloss)=\begin{cases}
                                                  C(\un_{\oa}) & (\loss=R_e) \\
                                                  C(\un_{\{\mathfrak{g}>0\}})
                                                               &
                                                                 (\loss=\I)
                                                \end{cases}
                              $ \\
    &$\mincostrange = [0, \Cinf(0)]$
                                           & $\maxcostrange = \begin{cases}
                         [\Csup(\maxloss),\maxcost] & (\loss=R_e) \\
                         [\Csup(\maxloss),c_0) \cup \{\maxcost\} & (\loss=\I)
                       \end{cases}
      $\\
    \midrule
    Opt.\ loss: properties
    &%
      \multicolumn{1}{l}{\noindent\begin{tabular}{@{}l}%
                       $\lossinf$ is continuous \\
                        $\lossinf$ is  decreasing on $[0, \Cinf(0)]$ \\
                       $\lossinf$ is constant on $[\Cinf(0), \maxcost]$ \\
                       $\lossinf(0)=\maxloss$ and  $\lossinf(\maxcost)=0$
                     \end{tabular}%
      }
    &
      \multicolumn{1}{l}{\begin{tabular}{@{}l}%
     $\lossup$ is continuous \\
     $\lossup$ is  decreasing on $[\Csup(\maxloss), c_0]$\\
    $ \lossup$ is constant on $[0, \Csup(\maxloss)]$ and on
      $[\cplus, \maxcost]$\\  
     $\lossup(0) =\maxloss$ and $\lossup(\maxcost)=0$
                     \end{tabular}%
      } \\
    \midrule
    
Inverse formula
    & $\lossinf\circ \Cinf =\Id$ on $[0, \maxloss]$
                            &\bothlosses{ $\lossup \circ \Csup =\Id$ on $[0, \maxloss]$}\\
    & $\Cinf\circ \lossinf =\Id$ on $\mincostrange$
    &  \bothlosses{ $\Csup\circ \lossup =\Id$
      on $\maxcostrange$} \\
    \midrule
Optimal strategies
    & \(
      \begin{aligned}
        \cp :=&  \set{C=\Cinf\circ \loss, \loss =  \lossinf \circ C } \\
             =&   \set{C=\Cinf\circ \loss} \\
             =&   \set{ \loss=\lossinf \circ C, C\in \mincostrange }
      \end{aligned}
      \)
    &\bothlosses{%
      \(\begin{aligned}
          \cpa :=& \set{C=\Csup\circ \loss} \cap \set{\loss = \lossup \circ C } \\
               =&  \set{C=\Csup\circ \loss} \\
               =&  \set{ \loss=\lossup \circ C, C\in \maxcostrange}
        \end{aligned}
      \)} \\
    & $\cp$ is compact &  $\cpa$ is
                         $\begin{cases} \text{compact} & (\loss=R_e) \\
                            \text{not compact}         & (\loss=\I)
                          \end{cases}
                         $ \\
    & $\un$ is Pareto opt. & 
                              $\left.\begin{aligned}
                                  &\un_{\oa} &(\loss=R_e) \\
                             & \un_{\{\mathfrak{g}>0\}}  &
                                                                 (\loss=\I)
                              \end{aligned}\, \right\}$ is
                              anti-Pareto opt. \\
    \midrule
 
{Optimal frontier}
    & \(\begin{aligned}
          \F :=&  (C, \loss)(\cp) \\
          =& (\Cinf, \Id)([0, \maxloss]) \\
          =& (\Id, \lossinf)(\mincostrange)
        \end{aligned}
      \)&
          \bothlosses{%
          \(\begin{aligned}
          \AF :=&  (C, \loss)(\cpa) \\
              =&   (\Csup, \Id)([0, \maxloss]) \\
    =& ( \Id, \lossup) (\maxcostrange)
            \end{aligned}\)
          }\\
    & $\F$ is connected and compact &  $\AF$ is
                                       $\begin{cases}
   \text{connected and compact} & (\loss = R_e) \\
   \text{not connected, not compact} &(\loss=\I)
                                        \end{cases}
                                       $\\
                                       \bottomrule
  \end{tabular}%
  
  \caption{%
    Summary of notation and results for  the  SIS  model and $\Delta=\DeltaSIS$.}
  \label{tab:not-PAP}
\end{table}

\begin{remark}[Result for other choices of $\Delta$]
  \label{rem:autre-DSIS}
  It is left to the reader to check that the statements given in
  Tab.~\ref{tab:not-PAP} hold also for the vaccinations strategies
  sets $\Delta_i\subset \DeltaSIS$, with
  $i\in \{\mathrm{Osc}, \mathrm{Block}, \mathrm{Ord}\}$, described in
  Section~\ref{sec:vaccination-strat}, but for the value of
  $\Csup(\maxloss)$ and the corresponding anti-Pareto vaccination
  strategies.  Indeed, the strategy $\eta_\mathrm{SIS}=\un_{\oa}$ for
  the loss $\loss=R_e$ or $\eta_\mathrm{SIS}=\ind{\{\mathfrak{g}>0\}}$
  for the loss $\loss=\I$ is no longer anti-Pareto optimal and it has
  to be replaced by
  $\eta_i=\min\{\eta\in \Delta_i\, \colon\, \eta\geq
  \eta_\mathrm{SIS}\}$ which is anti-Pareto optimal (in $\Delta_i$)
  and $C(\eta_i)=\Csup(\maxloss)$.

Thus, for the almost uniform vaccinations (that is, case
$i=\mathrm{Osc}$), we have that $\eta_\mathrm{Osc}=\max ((1-\delta)\un, \eta_\mathrm{SIS}))$ is   anti-Pareto   optimal   (in
$\DeltaOsc$), and  $\Csup(\maxloss)=C(\eta_\mathrm{Osc})$.
\end{remark}


\printbibliography
\end{document}